\documentclass[reqno,final]{amsart}
\usepackage{natbib}  
\usepackage{fancyhdr} 
\usepackage{color} 
\usepackage{hyperref} 
\usepackage{graphicx} 

\usepackage{pstricks}
\usepackage{amssymb}

\definecolor{aleacolor}{rgb}{0.16,0.59,0.78}

\hypersetup{
breaklinks,
colorlinks=true,
linkcolor=aleacolor,
urlcolor=aleacolor,
citecolor=aleacolor}

\renewcommand{\headheight}{11pt}
\renewcommand{\cite}{\citet}

\theoremstyle{plain}
\newtheorem{theorem}{Theorem}[section]

\newtheorem{corollary}[theorem]{Corollary}

\theoremstyle{definition}
\newtheorem{definition}[theorem]{Definition}
\theoremstyle{remark}
\newtheorem{remark}[theorem]{Remark}

\makeatletter \@addtoreset{equation}{section} \makeatother
\renewcommand\theequation{\thesection.\arabic{equation}}
\renewcommand\thefigure{\thesection.\arabic{figure}}
\renewcommand\thetable{\thesection.\arabic{table}}

\newcommand{\aleaIndex}[1]{\href{http://alea.impa.br/english/index_v#1.htm}{\bf #1}}
\eheader{ALEA}{\aleaIndex{0}}{2015}{00}{00}

\newcommand{\F}{\mathcal F}
\newcommand{\N}{\mathcal N}

\newcommand{\Z}{\mathbb Z}
\newcommand{\X}{\mathbf X}

\newcommand{\rt}{\rightarrow}
\newcommand{\lrt}{\longrightarrow}

\newcommand{\bc}{\begin{center}}
\newcommand{\ec}{\end{center}}

\newcommand{\ds}{\displaystyle}
\newcommand{\vphi}{\varphi}
\newcommand{\D}{\mathrm{d}}
\newcommand{\I}{\mathrm{i}}
\newcommand{\norm}[1]{\left\Vert#1\right\Vert}
\newcommand{\abs}[1]{\left\vert#1\right\vert}

\newcommand{\R}{\mathbb R}
\newcommand{\eps}{\varepsilon}

\newcommand{\E}{\mathrm{e}}

\begin{document}

\title[Distributional properties and parameters estimation of GSB process...]{Distributional properties and parameters estimation of GSB Process: An approach based on characteristic functions}

\author[V. Stojanovi\'c et al.]{Vladica Stojanovi\'c}
\author[]{Gradimir V. Milovanovi\'c}
\author[]{Gordana Jeli\'c}

\address{Faculty of Sciences and Mathematics,\newline
University of Kosovska Mitrovica\newline
Lole Ribara 29\newline
40\,000 Kosovska Mitrovica, Serbia.}
\email{vladica.stojanovic@pr.ac.rs}

\address[]{Serbian Academy of Sciences and Arts\newline
Kneza Mihaila 35\newline
11\,000 Belgrade, Serbia\newline
\& State University of Novi Pazar\newline
Vuka Karad\v zi\'ca bb\newline
36\,300 Novi Pazar, Serbia.\newline
\url{http://www.mi.sanu.ac.rs/~gvm/}}
\email{gvm@mi.sanu.ac.rs}

\address{Faculty of Technical Sciences,\newline
	University of Kosovska Mitrovica\newline
	Knez Milo\v sa 7\newline
40\,000 Kosovska Mitrovica, Serbia.}
\email{gordana.jelic@pr.ac.rs}

\thanks{The second author  was supported by the Serbian Ministry of Education, Science and Technological Development (No. \#OI\,174015).}


\begin{abstract}
 A general type of a Split-BREAK process with Gaussian innovations (henceforth, Gaussian Split-BREAK or GSB process) is considered. The basic stochastic properties of the model are studied and its characteristic function derived. A procedure to estimate the parameter of the GSB model based on the Empirical Characteristic Function (ECF) is proposed. Our simulations suggest that the proposed method performs well compared to a Method of Moment procedure used as benchmark. The empirical use of the GSB model is illustrated with an application to the time series of total values of shares traded at Belgrade Stock Exchange.
 \end{abstract}

\maketitle

\section{Introduction}
\label{sec:1}

The so-called STOchastic Permanent BREAKing (STOPBREAK) process, firstly introduced by \cite{STOPBREAK}, was successfully used in modeling  time series with ``large shocks", which have permanent, emphatic fluctuations in their dynamics. After this initial contribution, the STOPBREAK notion has been considered by several authors. For instance, in \citet{Diebold} or \citet{Gonzalo&Martinez} some modifications of the STOPBREAK process were discussed, while \citet{Huang&Fok} and \citet{Kapet&Tzav} investigated applications of the processes of this type. Finally, some new extensions in modeling structural breaks and ``large shocks" in empirical time series dynamics can be found, for instance, in \citet{Dendramis1,Dendramis2} or \citet{Hassler}. In those papers, different empirical applications of STOPBREAK-based models were investigated: economic stability of the oil market, relationships between daily closing market indexes of different countries, etc.

\citet{Split-BREAK,GSB1,GSB2} proposed an extension of the STOPBREAK process, where the so-called noise threshold indicator was set. The model obtained in this way, named the Split-BREAK process, represents a generalization of the basic STOPBREAK process, but also includes, in some special cases, other well-known time series models (see the following Section).
In this paper, the Gaussian distribution of the innovations of Split--BREAK process is assumed, and this process was named \textit{the Gaussian Split-BREAK $($GSB$)$ process}. A brief theoretical background, i.e., a definition and the basic stochastic properties of the process, are given in Section \ref{sec:2}. The main results of the paper are presented in the sections to follow. In Section \ref{sec:3}, we pay a special attention to the series of increments, named \emph{the Split--MA process}. First of all, a general explicit expression of an arbitrary order characteristic functions (CFs) of this process is given. Based on this, we investigate some distributional properties of the Split-MA process.
In Section \ref{sec:4}, we consider an estimation procedure of the Split-MA process parameters, using \textit{the Empirical Characteristic Functions $($ECF$)$ technique}. In the same section, we also study the asymptotic properties of these estimators. The numerical simulations of the ECF estimators are considered in Section \ref{sec:5}. An application of the GSB process and the ECF estimation procedure in modeling the dynamics of the total values of shares trading on Belgrade Stock Exchange are described in Section \ref{sec:6}. Finally, some conclusions are presented in Section \ref{sec:7}.

\section{GSB process. Definition and main properties}
\label{sec:2}

To begin with, we define a general form of the GSB process, as well as its basic stochastic properties.

\begin{definition}\label{GSB-def}\em Let $(Y_t)$, with $t\in\mathbb{Z}$ be a time series of random variables on some probability space $(\Omega, \F, P)$. Let  $F=(\F_t)$ be a
	filtration and denote by $(\eps_t)$ a sequence  of independent identical distributed (i.i.d.) Gaussian $\N(0,\sigma^2)$ random variables, adapted to the
	filtration $F$. We say that series $(Y_t)$ obeys a Gaussian Split-BREAK $(GSB)$ process, if and only if it satisfies
	\begin{equation}\label{GSB-Eq.1}
	A(L) Y_t=B(L)q_t \eps_t+C(L)(1-q_t)\eps_t,
	\end{equation}
	where
	\begin{equation}\label{Ind}
q_t \equiv	q_t(c)=I(\eps_{t-1}^2>c)=\left\{%
	\begin{array}{ll}
	1, & \hbox{$\eps_{t-1}^2>c$} \\[1mm]
	0, & \hbox{$\eps_{t-1}^2\leq c$},
	\end{array}%
	\right.
	\end{equation}
	is \textit{the Noise Indicator} of the (previous) realizations of $(\eps_t)$, $A(L)=1-\sum_{i=1}^m \alpha_i L^i$,
	$B(L)=1-\sum_{j=1}^n \beta_j L^j$, 	$C(L)=1-\sum_{k=1}^p \gamma_k L^k$, and $L$ is a back--shift operator.
\end{definition}

Notice that, according to (\ref{Ind}), $$E\big(q_t\eps_t|\F_{t-1}\big)=q_tE\big(\eps_t|\F_{t-1}\big)=0,$$
i.e. the sequence $(q_t\,\eps_t)$ is a martingale difference, as in the basic STOPBREAK model. Moreover, the values of the series $(q_t)$ determine the amount of participation of the previous elements of innovations $(\eps_t)$  in the series $(Y_t)$. The level of realizations of the innovations $(\eps_t)$ which will be statistically significant to be included in (\ref{Ind}) was determined with the parameter $c > 0$. In that way, the GSB process varies between the well known linear stochastic models  (see, for more details \citealp{GSB1, GSB2}). In the dependence of  $A(L),\;B(L)$ and $C(L)$ we have, for instance, the following processes:
$$\begin{array}{rll}
A(L)=B(L)=C(L)=1:&  Y_t=\eps_t&\textrm{(White
	Noise)}\\\vspace{.15mm}
A(L)=1, B(L)=C(L)\neq 1:&  Y_t=B(L)\eps_t
&\textrm{(MA model)}\\\vspace{.15mm}
A(L)\neq 1, B(L)=C(L)= 1:& A(L) Y_t=\eps_t&
\textrm{(AR model)}\\\vspace{.15mm}
A(L)\neq 1, B(L)=C(L)\neq 1:& A(L) Y_t=B(L)\eps_t&
\textrm{(ARMA model)}.\vspace{.15mm}
\end{array}$$

Subsequently, we suppose that $A(L)=C(L)\neq 1$ and $B(L)=1,$ when the model (\ref{GSB-Eq.1}) has the form
\begin{equation}\label{GSB-Eq.2}
Y_t-\sum_{j=1}^p \alpha_j\,
Y_{t-j}=\eps_t-\sum_{j=1}^p\alpha_j\,
\theta_{t-j}\,\eps_{t-j}, \qquad t\in \mathbb{Z},
\end{equation}
where $\alpha_j\geq0, \; j=1,\dots,p$ and
$\theta_t=1-q_t=I(\eps_{t-1}^2 \leq c)$ satisfy the non--triviality condition $b_c:=P\{\eps_{t-1}^2\leq c\}\in(0,1)$.
Thus, the representation (\ref{GSB-Eq.2}) can be considered as a specific non-linear ARMA model with the ``temporary" components $(\theta_{t-j}\eps_{t-j})$, which imply the specific structure of the GSB process.
It was proven in \citet{GSB1} that the necessary and sufficient conditions of strong stationarity of the series $(Y_t)$ are that the zeros $\lambda_1, \ldots, \lambda_p$ of the polynomial
$P(\lambda)=\lambda^p - \sum_{j=1}^p \alpha_j\;
\lambda^{p-j}$
satisfy the conditions
$|\lambda_j|<1$, $j=1,\dots, p$, i.e., that inequality $\sum_{j=1}^p \alpha_j < 1$ holds. In this case, the mean of the GSB series $(Y_t)$ is $E(Y_t)=0$, and its covariance function $\gamma_{_Y}(h)=E\big(Y_{t+h}\, Y_t\big),$ $h\geq 0$ satisfies
recurrence relation
\[\gamma_{_Y}(h)-\sum_{j=1}^p \alpha_j \Big[ \gamma_{_Y}(h-j)-s(h-j)I(h-j>0)\Big]=\left\{%
\begin{array}{ll}
\sigma^2, & \hbox{$h=0$},\\
0, & \hbox{$h\neq 0$},\\
\end{array}%
\right.
\]
where
\[s(h)=\left\{
\begin{array}{ll}
\sum_{j=1}^p \alpha_j\, s(h-j),& \hbox{$h\geq p$},\\[1mm]
\sum_{j=1}^h \alpha_j\, s(h-j), & \hbox{$h=1,\dots,p-1$},\\[1mm]
b_c\sigma^2, & \hbox{$h=0$}.\\[1mm]
\end{array}
\right.\]

On the other hand, similarly to the basic STOPBREAK process, the equality (\ref{GSB-Eq.2}) enables the additive decomposition $Y_t=m_t+\eps_t,$ where
\[
m_t=\sum_{j=1}^p \alpha_j \big(Y_{t-j}-\theta_{t-j}
\eps_{t-j}\big)=\sum_{j=1}^p \alpha_j \big(m_{t-j}+q_{t-j}
\eps_{t-j}\big)
\]
is a sequence of random variables, called as \emph{the
martingale means}. In this way, the series $(m_t)$ represents a generalization of an analogue series in \cite{STOPBREAK}. It can be easily shown that $(m_t)$ satisfies the same stationarity conditions as the series $(Y_t)$.  Moreover, regardless to stationary of the series $(Y_t)$ and $(m_t)$, it follows
\begin{equation}\label{PSB3}
E\big(Y_t \mid \F_{t-1}\big)=m_t+E\big(\eps_t \mid
\F_{t-1}\big)=m_t, \quad t\in\mathbb Z,
\end{equation}
and based on that, we have
$ E(Y_t) = E(m_t)=\mu(= {\rm const})$. In a similar way, the variance of GSB process can be obtained. According to the equality
\begin{equation}\label{PSB4}
Var\big(Y_t \mid \F_{t-1}\big)=E\big(Y_t^2 \mid \F_{t-1}\big)-m_t^2=\sigma^2,
\end{equation}
the conditional variance of $(Y_t)$ is constant and equal to the variance of $(\eps_t)$, and the variances of $(m_t)$ and $(Y_t)$ satisfy the relation $Var(Y_t)=Var(m_t)+\sigma^2$. Note that the equalities (\ref{PSB3}) and (\ref{PSB4}) can explain the stochastic behavior of the series $(Y_t)$. Namely, the sequence $(m_t)$ is predictable, and it represents a  stability component of the process $(Y_t)$. On the other hand, realizations of the innovations $(\eps_t)$ represent the random fluctuations around the values $(m_t)$ (see Figure \ref{Fig.1}, left).

We next describe the stochastic
structure of another time series, the so--called increments $X_t:=A(L) Y_t$, $t\in \mathbb{Z}$. According to (\ref{GSB-Eq.2}), this series can be written as
\begin{equation}\label{Xt}
X_t=\eps_t-\sum_{j=1}^p \alpha_j \,\theta_{t-j} \eps_{t-j},
\quad t \in \mathbb{Z}.
\end{equation}
Therefore, $(X_t)$ has the multi-regime structure, depending on the realizations of indicators $(\theta_t)$. If all squared innovations $\eps_{t-j}$ are sufficiently large (i.e. greater than $c$), an increment $X_t$ will be equal to $\eps_t$. On the other hand, squared innovations, which do not exceed the critical value $c$, produce a ``part of" MA$(p)$ representation of the series $(X_t)$ (Figure \ref{Fig.1}, right). Due to this, $(X_t)$ is named \emph{the $($Gaussian$)$ Split-MA model $($of order p$)$}, or simply \emph{the Split-MA$(p)$ model}.

Note that $(X_t)$ is a stationary process, with the mean $E(X_t)=0$ and the covariance
\begin{equation}\label{Cov X}
\gamma_{_X} (h)=E(X_tX_{t+h})=\left\{%
\begin{array}{ll}
\sigma^2\Big(1+b_c\sum_{j=1}^p \alpha_j^2\Big), & \hbox{$h=0$}, \\[1.5mm]
\sigma^2\,b_c\Big(\sum_{j=1}^{p-h}\alpha_j\,\alpha_{j+h}-\alpha_{h}\Big), & \hbox{$1\leq h\leq p-1$},\\[1.5mm]
-\sigma^2\,b_c\alpha_p\,, & \hbox{$h=p$},\\[1.5mm]
0, & \hbox{$h>p$}. \end{array}%
\right.
\end{equation}
In \citet{GSB1} it was shown that the process $(X_t)$ is invertible if and only if the zeros $r_1$, $\ldots$, $r_p$ of the polynomial
$
Q(\lambda)=\lambda^p-b_c\sum_{j=1}^p \alpha_j \lambda^{p-j}
$
meet the condition $|r_j|<1$, $j=1,\dots p$, i.e., the inequality
$b_c \sum_{j=1}^p\alpha_j<1$ holds. Then,
\begin{equation}\label{Invert}
\eps_t=\sum\limits_{k=0}^\infty \omega_{k}(t)\,X_{t-k}, \quad t\in
\Z,
\end{equation}
where
\begin{equation}\label{Omega}
\omega_{k}(t)=\left\lbrace
\begin{array}{ll}
\theta_{t-k}\sum_{j=1}^p \alpha_j\,
\omega_{k-j}(t), & k\geq p,\\
\theta_{t-k}\sum\limits_{j=1}^k \alpha_j \,
\omega_{k-j}(t), & 1\leq k \leq p-1,\\
1, & k=0.
\end{array}
\right.
\end{equation}
\citet{GSB1} proved that the representation $(\ref{Invert})$ is almost surely unique, as well as that the sum on the right side converges with the probability one and in the mean--square.

According to the aforementioned facts, the conditions of invertibility of increments $(X_t)$ are weaker than the stationary conditions of the series $ (Y_t) $ and $ (m_t) $.
This is particularly interesting in the case of the so--called
{\it integrated $($standardized$)$} time series, where
$
\sum_{j=1}^p \alpha_j=1.
$
Further, we take into consideration time series of this type, for the following reasons:

(i) In this case, the series $(Y_t)$ and $(m_t)$ are non--stationary processes, with the non--zero mean.
These properties are typical in dynamics of the time series with ``large shocks", as well as in their practical applications.

(ii) If the parameter $b_c$  takes a non--trivial values, i.e., $b_c \in(0,1)$, the series $(X_t)$ will be stationary and invertible. Then, the whole estimation procedure of unknown parameters, as we see further, will be based just on its realizations.

(iii) Finally, the assumption
$ \sum_{j=1}^p \alpha_j=1 $ is fully in line with the definition of basic STOPBREAK process in \cite{STOPBREAK}, where $p = \alpha_1 = 1$.

Figure  \ref{Fig.1} illustrates the dynamics of the kind of time series defined above, in the case of the simplest, the GSB(1) and Split-MA(1) processes, where the model's parameters are $\alpha_1=c=\sigma=1$.\\

\begin{figure}[hbtp]
	\includegraphics[width=1\textwidth]{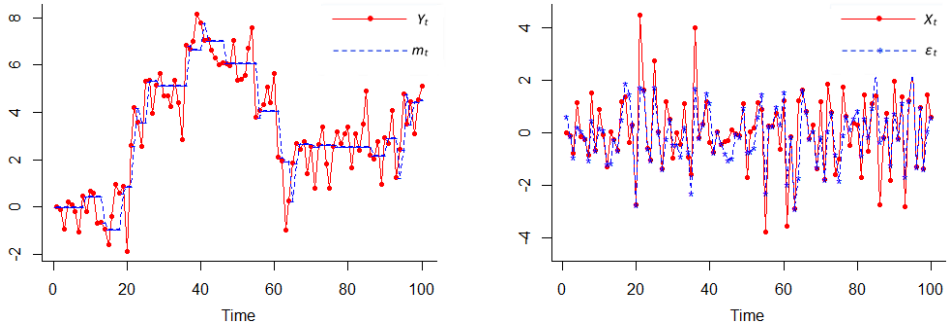}
	\caption{Comparative graphs of the basic GSB series.}\label{Fig.1}
\end{figure}

\section{Distributional properties of Split--MA process}
\label{sec:3}

In this section, we consider some distributional properties of increments $(X_t)$. For this purpose, we define the parameter vector  $\theta=(\alpha_1,\dots,\alpha_p,b_c, \sigma^2)'$. Note that the critical value $c$ can be easily obtained as $c=\sigma^2F^{-1}(b_c)$, where $F(x)$ is the cumulative distribution function (CDF) of $\chi_1^2$ distributed random variable. We now introduce the following
definition:
\begin{definition}
Let $\mathbf{u}=(u_1,\dots,u_\ell)'\in\mathbb{R}^\ell$ and $\X_t^{(l)}:=(X_t,\dots, X_{t+\ell-1})'$, $t\in \Z$, be the overlapping blocks of the process $(X_t)$. \textit{The $\ell$--dimensional CF} of the random vector $\X_t^{(\ell)}$ is
\begin{equation}\label{Eq.3.1}
	\varphi_X^{(\ell)}(\mathbf{u};\theta):=E\left[\exp\left(\I\mathbf{u}'\X_{t}^{(\ell)} \right)\right]=E\Bigg[\exp\bigg(i\sum_{j=1}^\ell u_j X_{t+j-1}\bigg)\Bigg].
\end{equation}
\end{definition}

The following statement gives an explicit expression of the CF of Split--MA$(p)$ process.
\begin{theorem}\label{Thm:1}
	Let $(X_t)$ be the Split-MA(p) process defined by $(\ref{Xt})$. Then, CF of the order $\ell\in\mathbb{N}$ of the random process $(\X_t^{(\ell)})$ is given by
\begin{eqnarray}\label{Eq.3.2}
 \varphi_X^{(\ell)}(\mathbf{u};\theta)&=&
\exp\bigg(\frac{-\sigma^2u_M^2}{2}\bigg)\times
\prod_{j=1}^{M-1}\Bigg[(1-b_c)\exp\bigg(\frac{-\sigma^2 u_j^2}{2}\bigg)
\nonumber\\
&&
 +b_c \exp\bigg(\frac{-\sigma^2}{2}\bigg(u_j-\sum_{k=1}^{M-j}\alpha_k u_{k+j}\bigg)^2\bigg)\Bigg]\nonumber\\
&&\times \prod_{j=1}^{M}\Bigg[1-b_c+b_c \exp\bigg(\frac{-\sigma^2}{2}\bigg(\sum_{k=j}^{M}\alpha_k u_{k-j+1}\bigg)^2\bigg)\Bigg],
\end{eqnarray}
where $M=\max\{p,\ell\}$ and $u_{\ell+j}=\alpha_{p+k}=0$, $j=\ell+1,\dots,M$, $k=p+1,\dots,M$.
\end{theorem}
\begin{proof}
	According to the definition of the Split--MA process, as well as the CF (\ref{Eq.3.1}), it is valid that
	$\varphi_X^{(\ell)}(\mathbf{u};\theta)=E\big[\mathcal{L}_X^{(\ell)}(\mathbf{u};\theta)\big]$, where
\begin{eqnarray*}
\mathcal{L}_X^{(\ell)}(\mathbf{u};\theta)&=&
\exp\left\{\I\Bigg[u_M\eps_{t+M-1}+\sum\limits_{j=1}^{M-1}\bigg(u_j-\theta_{t+j-1}\sum_{k=1}^{M-j}\alpha_k u_{k+j}\bigg)\eps_{t+j-1}\right.\\
&&-\left.\sum\limits_{j=1}^{M}\eta_{t-j}\sum_{k=j}^{M}\alpha_k u_{k-j+1}\Bigg]\right\},
\end{eqnarray*}
and, for an arbitrary $t\in\Z$, we denote $\eta_t:=\theta_t \eps_t$. As the random variables $\theta_t$ and $\eps_t$ are independent, it follows that $E(\eta_t)=0$, $\mathrm{Var}(\eta_t)=E(\eta_t^2)=b_c \sigma^2$ and $\mathrm{Cov}(\eta_t,\eta_{t+k}) = 0$, $k\neq0$. Thus, $(\eta_t)$ is a series of uncorrelated random variables, with the CF
	\begin{equation}\label{Eq.3.3}
	\vphi_\eta(u;\theta)=\int_{-\infty}^{+\infty}\E^{iux}\big[b_c F_\eps+(1-b_c)F_0\big](\D x),
	\end{equation} 	
where $F_\eps(x)=P\{\eps_t<x\}$ and $F_0(x)=I\{x>0\}$ are the CDFs of the random variables $\eps_t:\N(0,\sigma^2)$ and $I_0\stackrel{as}=0$, respectively. As appropriate CFs of these random variables are $\vphi_\eps(u)=\E^{-\sigma^2u^2/2}$ and $\vphi_0(u)\equiv1$, a substitution in (\ref{Eq.3.3}) gives
	\begin{equation}\label{Eq.3.4}
	\vphi_\eta(u;\theta)=(1-b_c)\vphi_0(u)+b_c\, \vphi_\eps(u)=1+b_c\left(\E^{-\sigma^2u^2/2}-1\right).
	\end{equation} 	
According to this, it is obvious that   $\overline{\vphi_\eta(u;\theta)}=\vphi_\eta(u;\theta)$, and the CF $\varphi_X^{(\ell)}(\mathbf{u};\theta)$ can be rewritten in the form
\begin{eqnarray*}
	\vphi_X^{(\ell)}(\mathbf{u};\theta) &=&\vphi_\eps(u_M)\times \prod\limits_{j=1}^{M-1}\Bigg[(1-b_c) \vphi_\eps(u_j)+b_c\vphi_\eps\bigg(u_j-\sum_{k=1}^{M-j}\alpha_k u_{k+j}\bigg)\Bigg]\\
	&&\times\prod\limits_{j=1}^{M}\Bigg[\vphi_\eta\bigg(\sum_{k=j}^{M}\alpha_k u_{k-j+1}\bigg)
	\Bigg].
\end{eqnarray*}
The last equality and Eq.(\ref{Eq.3.4}) imply Eq.(\ref{Eq.3.2}), i.e., complete the proof of this theorem.
\end{proof}

\begin{remark}
	According to the previous theorem, the first two orders CFs of $(X_t)$ are
	\begin{eqnarray}
	\label{Eq.3.5}	\varphi_X^{(1)}(u;\theta)&=&\E^{-\sigma^2 u^2/2}\prod_{j=1}^{p}\left[1+b_c\left(\E^{-\alpha_j^2\sigma^2 u^2/2}-1\right)\right],\\[1mm]
	\label{Eq.3.6}		\varphi_X^{(2)}(u_1,u_2;\theta)&=&\E^{-\sigma^2u_2^2/2}\left[(1-b_c)\E^{-\sigma^2u_1^2/2}+b_c\E^{-\sigma^2 (u_1-\alpha_1u_2)^2/2}\right]\\
	\nonumber		&&\times{\prod_{j=1}^{p}\left[1+b_c\left(\E^{-\sigma^2(\alpha_{j}u_1+\alpha_{j+1}u_2)^2/2}-1\right)\right]},
	\end{eqnarray}
respectively, where $\alpha_{p+1}=0$.
\end{remark}

The CFs in Eqs.(\ref{Eq.3.5})--(\ref{Eq.3.6}) can be used to study the stochastic properties of the series  $(X_t)$ and for parameters estimation. Note that the first order CF (3.5) and Levy's convergence theorem immediately imply:
\begin{corollary}
	The CDF of the random variables $X_t$, $t\in\Z$ is
	\[
	F_X(x)=\bigotimes\limits_{j=1}^p \big[b_c F_j +(1-b_c)F_0\big]\otimes F_\eps(x),
	\]
	where $``\otimes"$ denotes the convolution operator and $F_j(x)$, $j=1,\dots,p$, are the CDFs of the random variables with the Gaussian distribution $\N(0,\alpha_j^2\sigma_j^2)$.
\end{corollary}
Moreover, the following proposition gives a recurrence relation for moments of the Split--MA process.
\begin{theorem}\label{Thm:2}
	For an arbitrary $n\in \mathbb N$, the Split--MA$(p)$ processs $(X_t)$ has the finite $n$-th moment
	\[
	E\left(X_t^n\right)=\left\{\begin{array}{ll}
	0,& n=2m-1,\\[2mm]
	\ds\sum\limits_{k=1}^{m} {2m-1\choose 2k-1} (2k-1)!!\ \sigma^{2k}\ W_k(b_c)\ E\left(X_t^{2(m-k)}\right),& n=2m,
	\end{array}\right.
	\]
	where $(2k-1)!! = (2k-1)(2k-3)\cdots 2\cdot 1$,
	\[
	W_k\left(b_c\right)=\left\{\begin{array}{ll}
	1+b_c\sum_{j=1}^{p}\alpha_j^2,& k=1,\\[2mm]
	L_k(b_c)\sum_{j=1}^{p} \alpha_j^{2k},& k=2,3,\dots,
	\end{array}\right.
	\]
	and $\{L_k(x)\}_{k\in\Z}$ is a sequence of algebraic polynomials defined by the generating function $$G(x,t)=\log(1-x+x\E^t)=\sum_{k=1}^{+\infty}L_k(x) \frac{t^k}{k!}.$$
\end{theorem}
\begin{proof}
	Let us denote $\Psi(u;\theta):=\log\vphi_X^{(1)}(u;\theta)=-\sigma^2u^2/2+\sum_{j=1}^{p}\log f_j(u)$, where $f_j(u):=1+b_c\left(\E^{-\alpha_j^2\sigma^2u^2/2}-1\right)$, $j=1,\dots,p$. According to Leibniz's formula, the $n$-th derivative of the CF $\vphi_X^{(1)}(u;\theta)$ can be written as	
	\[
	\frac{\D^n \vphi_X^{(1)}(u;\theta)}{\D u^n}=\sum_{k=1}^{n} {n-1 \choose k-1}\ \frac{\D^{n-k} \vphi_X^{(1)}(u;\theta)}{\D u^{n-k}}\ \frac{\D^k \Psi(u;\theta)}{\D u^k}.
	\]
	Since $\vphi_X^{(1)}(0;\theta)=1$ and $f'_j(0)=0$, $j=1,\dots,p$, we conclude that $\ds\frac{\D }{\D u}\vphi_X^{(1)}(0;\theta)=0$ and  $\ds\frac{\D }{\D u}\Psi(0;\theta)=0$. Furthermore, one can easily prove that
	\[\frac{\D^{2m-1} \vphi_X^{(1)}(0;\theta)}{\D^{2m-1} u}=\frac{\D^{2m-1} \Psi(0;\theta)}{\D^{2m-1} u}=0, \]
	for each $ m\in \mathbb N.$ On the other hand, for the derivatives of even orders we have
	\[
	\frac{\D^{2m} \vphi_X^{(1)}(0;\theta)}{\D u^{2m}}=\sum_{k=1}^{m} {2m-1 \choose 2k-1}\ \frac{\D^{2(m-k)} \vphi_X^{(1)}(0;\theta)}{\D u^{2(m-k)}}\ \frac{\D^{2k} \Psi(0;\theta)}{\D u^{2k}},
	\]
	i.e.,
	\begin{equation}\label{Eq.3.7}
	U_m=\sum_{k=1}^{m}{2m-1 \choose 2k-1} U_{m-k}V_k,
	\end{equation}
	where we set $U_k:={\D^{2k} \vphi_X^{(1)}(0;\theta)}/{\D u^{2k}}$
	and $V_k:={\D^{2k} \Psi(0;\theta)}/{\D u^{2k}}$. Using the induction method, it can be proven that
	\[
	V_k=(-1)^k (2k-1)!!\ \sigma^{2k} W_k(b_c), \quad k \in \mathbb{N}.
	\]
	Then, substituting $V_k$ in (\ref{Eq.3.7}) and using that $E(X_t^n)=i^{-n} \D^n \vphi_X^{(1)}(0;\theta) / \D u^n$, the statement of this theorem follows immediately.
\end{proof}

\begin{remark}
	According to Theorem \ref{Thm:2}, we can simply obtain the kurtosis
	\[
	K_X:= \frac{E\left(X_t^4\right)}{\left[E\left(X_t^2\right)\right]^2} = 3\left[1+\frac{W_2(b_c)}{\left(W_1(b_c)\right)^2}\right]\geq 3.
	\]
	It is obvious that the equality $K_X=3$ holds if and only if
	\begin{equation}\label{Eq.3.8}
	b_c=1\;(c=+\infty) \quad \vee \quad b_c=c=0 \quad \vee \quad \alpha_1=\cdots=\alpha_p=0.
	\end{equation}
	In the first case, the Split--MA$(p)$ model is reduced (almost surely) to the linear MA model, while the other two cases give the Gaussian innovations $(\eps_t)$. Thus, (\ref{Eq.3.8}) represents the necessary and sufficient conditions for the process  $(X_t)$ to have a Gaussian distribution. In general, for the non--trivial values $b_c\in(0,1)$ and under previously assumed condition $\sum_{j=1}^{p}\alpha_j=1$, it will be $K_X>3$. Then, the random variables $X_t$ have a non--Gaussian distribution ``peaked" at $E(X_t)=0$.
	However, the higher order Split--MA processes  approximately have the Gaussian distribution, under condition
	\[
	\frac{W_2(b_c)}{\left(W_1(b_c)\right)^2}\lrt0,\quad p\rt+\infty.
	\]
Typical situation of this kind can be seen in Figure \ref{Fig.2}, where the kurtosis of the  Split--MA$(p)$ processes, as the functions of $b_c\in(0,1)$, and with equidistant coefficients $\alpha_1=\cdots=\alpha_p=2^{-k}$, $p=2^{k}$, $k=0,1,2,3$ are shown. In this case, after some simple computations, we find that
	\[ K_X=3\left[1+\frac{b_c(1-b_c)}{2^k\left(2^k+b_c\right)^2}\right]\lrt 3,\quad k\rt+\infty.
	\]
\end{remark}
\begin{figure}[htbp]
	\begin{center}
		\includegraphics[width=.65\textwidth]{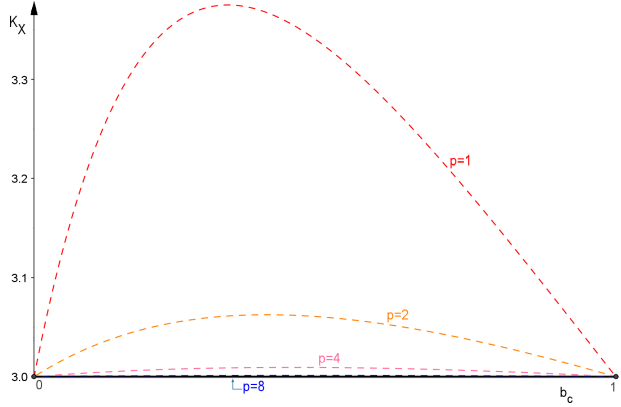}
		\caption{Kurtosis of Split--MA$(p)$ processes with  equidistant coefficients of order $p=2^k$, $k=0, 1,2,3$.}\label{Fig.2}
	\end{center}
\end{figure}

In a more precise manner, the following statement provides
some sufficient conditions for the asymptotic normality of the Split--MA process of infinite order.

\begin{theorem}\label{Thm:3}
	Let $X_t^{(\infty)}:=\eps_t-\sum_{j=1}^{+\infty}\alpha_j \theta_{t-j}\eps_{t-j}$, $t\in \Z$, be the Split--MA$(\infty)$ process, with $\sum_{j=1}^{+\infty}\alpha_j=1$ and $\alpha_j>0$ for the infinite number of $\alpha_j$. If for any $p\in\mathbb{N}$ the condition
	\begin{equation*}
	M_p:=\frac{\max\{\alpha_1^2,\dots,\alpha_p^2\}}{\min\{\alpha_1^2,\dots,\alpha_p^2\}}=\mathrm{O}\left(p^\delta\right), \quad 0\leq\delta<1/2
	\end{equation*}
	holds, then the random variables $X_t^{(\infty)}$ have the Gaussian distribution $\N(0,\Sigma^2)$, where $\Sigma^2=\sigma^2\left(1+b_c\sum_{j=1}^{+\infty}\alpha_j^2\right)$.
\end{theorem}

\begin{proof}
	Similarly to the proof of Theorem \ref{Thm:1}, denote $\eta_t:=\theta_t\eps_t$, $t\in\Z$. Hence, the equalities $E(\alpha_j\eta_{t-j})=0$ and $E(\alpha_j\eta_{t-j})^2=b_c \alpha_j^2  \sigma^2$ hold. According to
	\[
	\sum\limits_{j=1}^{+\infty}\alpha_j^2\leq \Bigg(\sum\limits_{j=1}^{+\infty}\alpha_j\Bigg)^2=1,
	\]
	it follows $\sum_{j=1}^{+\infty}\mathrm{Var}(\alpha_j\eta_{t-j})\leq b_c\sigma^2.$
	Therefore, the sum $\sum_{j=1}^{+\infty}\alpha_j\eta_{t-j}$ converges almost surely, i.e., the process $\left(X_t^{(\infty)}\right)$ is well defined.
	
	Further on, for an arbitrary $p\in \mathbb{N}$, denote $X_t^{(p)}:=\sum_{j=1}^{p}\alpha_j\eta_{t-j}$, $t\in\Z$, and $s_p^2:=\mathrm{Var}\left(X_t^{(p)}\right)=b_c\sigma^2\sum_{j=1}^{p}\alpha_j^2$. We then have
	\begin{eqnarray*}
		0&\leq& \frac{1}{s_p^4}\sum_{j=1}^{p}E(\alpha_j\eta_{t-j})^4=\frac{3}{b_c}
		\frac{\sum_{j=1}^{p}\alpha_j^4}{\left(\sum_{j=1}^{p}\alpha_j^2\right)^2}=
		\frac{3}{b_c} \left(1+ \frac{\sum_{j=1}^{p}\sum_{k\neq j}\alpha_j^2\alpha_k^2}{\sum_{j=1}^{p}\alpha_j^4}\right)^{-1}\\	
		&\leq& \frac{3}{b_c} \left(1+\frac{p-1}{M_p^2}\right)^{-1}=\mathrm{O}\left(p^{2\delta-1}\right)\lrt0, \quad p\rt+\infty,	
	\end{eqnarray*}
	i.e. Lyapunov's condition holds. Thus, $s_p^{-1} X_t^{(p)}\stackrel{d}\lrt\N(0,1)$, $p\rt+\infty$, and according to this, it is easy to obtain the statement of this theorem.
\end{proof}
The previous theorem shows that probability distributions of the large--order Split--MA processes, with ``small variation" of coefficients $\alpha_j$, can be approximated with a Gaussian distribution. The CF of the standard Gaussian distribution $\N(0,1)$ is shown in Figure \ref{Fig.3}, and compared to the CFs of Split--MA$(p)$ processes with the equidistant coefficients $\alpha_1=\cdots=\alpha_p=2^{-k}$, where $p=2^k$, $k=1,2\dots$ and $c=\sigma=1$. In this case, it is obvious that $\delta=0$, and therefore, a normal approximation of distributions of the Split--MA series of larger order $p\in\mathbb{N}$ is valid. On the other hand, it is clear that the Split--MA models with a ``small" order $p$ have a pronounced ``peaked" distribution, that significantly differs from the Gaussian one. Therefore, in addition to the general model, we will further investigate, in more details, an estimation procedure of the simplest Split--MA process, when $p=\alpha_1=1$.
\begin{figure}[htbp]
	\begin{center}
		\includegraphics[width=.65\textwidth]{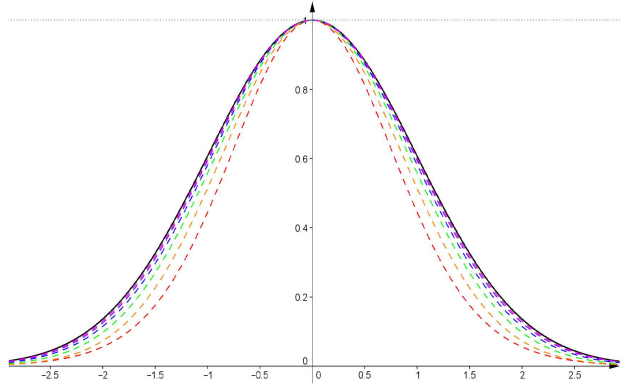}
		\caption{The CF of standard Gaussian distribution (solid line), compared to the CFs of Split--MA$(p)$ processes with  equidistant coefficients of order $p=2^k$ (dashed lines).}\label{Fig.3}
	\end{center}
\end{figure}

\section{Estimations of parameters by the ECF method}
\label{sec:4}

Parameters estimation procedure of the standard STOPBREAK model, introduced by Engle and Smith \citet{STOPBREAK}, was mainly based on the quasi--maximum likelihood (QML) method. In  case of the Split--BREAK model, it can be proven that the likelihood function is unbounded at the origin, and it disables the usage of the QML, as well as some closely related methods, based on the maximum likelihood approach. For these reasons, \citet{Split-BREAK,GSB2} have used some distributional independent  estimation methods, that were based on two typical non--parametric procedures: methods of moments and Gauss--Newton's regression for non-linear functions. However,  the main problem in realization of these procedures is non-observability of the Split--MA$(p)$ process $X_t:=Y_t-\sum_{j=1}^p \alpha_j Y_{t-j}$, for $p>1$.

Here we describe the new parameters estimation procedure for this process, based on \textit{the Empirical Characteristic Function (ECF) method}. In the time series analysis, this method was described for the first time in  \citet{Feuerverger}, as well as in \citet{Knight&Satchell1,Knight&Satchell2}. After that, several authors (cf. \citealp{Singleton,Knight et al., Knight&Yu}) described implementation of the ECF method in econometric analysis and finance. On the other hand, several other new theoretical extensions of the CF--based estimators can be found, for instance in \citet{Balak.} and \citet{Kotchoni1,Kotchoni2}. Here we apply a similar procedure as in parameters estimation of the so--called Split--SV model, described in \citet{GVM3} and \citet{Split-SV}.

The main aim of the ECF method is to minimize ``the distance" between the theoretical CF and the appropriate ECF of some stochastic model. In case of the Split--MA process we denote as $\X_T:=\{X_1,\dots,X_T\}$ a realization of the series $(X_t)$ of length $T\in\mathbb{N}$. Then the appropriate $\ell$-dimensional ECF of the random sample $\X_T$ is
\[
\widetilde{\varphi}_T^{(\ell)}(\mathbf{u}) :=\frac{1}{T-\ell+1}\sum_{t=1}^{T-\ell+1} \exp\left(\I\mathbf{u}'\X_{t}^{(\ell)} \right),
\]
and the objective function is
\begin{equation}\label{Eq.4.1}
S_T^{(\ell)}(\theta):=\idotsint\limits_{\R^\ell} g(\mathbf{u})\abs{\varphi_X^{(\ell)}(\mathbf{u};\theta)-\widetilde{\varphi}_T^{(\ell)}(\mathbf{u})}^2 \D \mathbf{u},
\end{equation}
where $\vphi_X^{(\ell)}(\mathbf{u};\theta)$ is the CF of order $\ell$, defined by Eq. (\ref{Eq.3.1}), $\D \mathbf{u}:= \D u_1\cdots \D u_\ell$ and $g:\mathbb R^\ell\rt\R^+$ is a some weight function. The estimates based on the ECF method are thus obtained by a minimization of the objective function (\ref{Eq.4.1}) with respect to the parameter $\theta=(\alpha_1,\dots,\alpha_p,b_c,\sigma^2)'$. More precisely, they represent the solutions of the minimization equation
\begin{equation}\label{Eq.4.2}
\hat{\theta}_T^{(\ell)}=\mbox{arg}\min_{\theta\in\Theta}S_T^{(\ell)}(\theta),
\end{equation}
where $\Theta=[0,1]^p\times (0,1)\times(0,+\infty)$ is the parameter space of the non--trivial, stationary and invertibile Split--MA$(p)$ process. One of the important problems here is the choice of block size $\ell\in\mathbb N$. We refer to \citet{Knight&Yu}, where the optimal values of $\ell$ were discussed in order to achieve asymptotic efficiency of the ECF estimators for some linear Gaussian time series. Similarly to them, we investigate the strong consistency and asymptotic normality (AN) of the ECF estimates of Split--MA model's  parameters, under certain necessary conditions.

\begin{theorem}\label{Thm:4}
	Let $\theta_0$ be the true value of the parameter $\theta$, and for an arbitrary $T=1,2,\ldots$, let $\hat{\theta}_T^{(\ell)}$ be solutions of the equation $(\ref{Eq.4.2})$. In addition, let us  suppose that the following regularity conditions are fulfilled:
	\begin{itemize}
		\item [\rm(i)]
		There exists the set $\Theta'=[0,1]^p\times(0,1)\times(0,M_{\sigma^2})\subset \Theta$, where $M_{\sigma^2}$ is chosen sufficiently large so that $\theta_0, \hat{\theta}_T^{(\ell)}\in\Theta'$ for all $T \geq T_0>0$;
		\vskip1mm
	\item [\rm(ii)]
		$\dfrac{\partial^2
			S_T^{(\ell)}(\theta_0)}{\partial \theta\ \partial
			\theta'}$ is a regular matrix;
		\vskip1mm
	\item [\rm(iii)]
		$\dfrac{\partial \varphi_X^{(\ell)}(\mathbf{u};\theta_0)}{\partial \theta}\ \dfrac{\partial \varphi_X^{(\ell)}(\mathbf{u};\theta_0)}{\partial \theta'}$ is a non-zero matrix, uniformly bounded by the strictly positive, $g$-integrable function $h:\R^\ell\to \R^+$.
	\end{itemize}
	Then,  $\hat{\theta}_T^{(\ell)}$ is
	strictly consistent and asymptotically normal estimator for $\theta$, for any $\ell\geq p+1$.
\end{theorem}

\begin{proof}
In order to prove the consistency of $\hat{\theta}_T^{(\ell)}$, we check the sufficient consistency conditions of extremum estimators (see, for instance \citealp{Newey}). Note that, under assumption (i), the set $\overline{\Theta'}=[0,1]^p\times[0,1]\times[0,M_{\sigma^2}]$ is a compact, and $\theta_0\in {\rm int}(\overline{\Theta'})$.
As the series $(X_t)$ is ergodic and
$\widetilde{\varphi}_T(\mathbf{u})$ is an unbiased estimator of $\varphi_X^{(\ell)}(\mathbf{u};\theta_0)$, the strong law of large numbers gives $\widetilde{\varphi}_T(\mathbf{u})\stackrel{\rm as}\lrt \varphi_X^{(\ell)}(\mathbf{u};\theta_0)$, and hence $\sup_{\theta\in\overline{\Theta'}} \abs{\widetilde{\varphi}_T(\mathbf{u})-\varphi_X^{(\ell)}(\mathbf{u};\theta_0)}\stackrel{\rm as}\lrt 0$, when
$T\rt +\infty$.

Further, if we define the function
\[
S_0^{(\ell)}(\theta):=\idotsint \limits_{\R^\ell}g(\mathbf{u})
\abs{\varphi_X^{(\ell)}(\mathbf{u};\theta_0)-\varphi_X^{(\ell)}(\mathbf{u};\theta)}^2 \D\mathbf{u}\geq0,
\]
then, according to Eq.(\ref{Eq.3.2}), it is obvious that $S_0^{(\ell)}(\theta)$ is continuous on $\Theta'$. Using the same deliberation as in case of the standard MA processes (see, for instance \citealp{Yu}), it is easy to see that, under assumption $\ell\geq p+1$, the equality $S_0^{(\ell)}(\theta)=0$ holds only if $\theta=\theta_0$. Thus, $S_0^{(\ell)}(\theta)$ has an unique minima at the true parameter value, and in the same way as in \citet{Knight&Yu}, it can be proven that
\[
\abs{S_T^{(\ell)}(\theta)-S_0^{(\ell)}(\theta)}\leq 4 \idotsint \limits_{\R^\ell}g(\mathbf{u})\abs{\widetilde{\varphi}_T^{(\ell)}(\mathbf{u})-\varphi_X^{(\ell)}(\mathbf{u};\theta_0)}	\D\mathbf{u}.
\]
Hence, $\sup_{\theta\in\overline{\Theta'}} \abs{S_T^{(\ell)}(\theta)-S_0^{(\ell)}(\theta_0)}\stackrel{\rm as}\lrt 0$,
$T\rt +\infty$, i.e., $S_T^{(\ell)}(\theta)$ uniformly converges almost surely to $S_0^{(\ell)}(\theta)$.
Consequently, according to Theorem 2.1 in \citet{Newey},  $\hat{\theta}_T^{(\ell)}-\theta_0\stackrel{\rm as}\lrt 0$, when
$T\rt +\infty,$ i.e.,
the estimator $\hat{\theta}_T^{(\ell)}$ is strictly consistent.

In order to show the AN, note that the function $S_T^{(\ell)}(\theta)$ has continuous partial derivatives up to the second order, for any component of the vector $\theta$. Thus, the Taylor expansion of $\partial S_T^{(\ell)}(\theta)/\partial \theta$ at $\theta=\theta_0$ gives
	\begin{equation*}
	\frac{\partial S_T^{(\ell)}(\theta)}{\partial \theta}= \frac{\partial S_T^{(\ell)}(\theta_0)}{\partial \theta}+
	\frac{\partial^2 S_T^{(\ell)}(\theta_0)}{\partial \theta\ \partial
		\theta'}\cdot(\theta-\theta_0)+{o}(\theta-\theta_0).
	\end{equation*}
	For a sufficiently large $T$, we substitute $\theta$ by $\hat{\theta}_T^{(\ell)}$, under assumption {\rm(ii)} and the fact that
	$\partial S_T^{(\ell)}(\hat{\theta}_T^{(\ell)}) /\partial \theta=0$. Then we have
	\begin{equation*}
	\hat{\theta}_T^{(\ell)}-\theta_0=-\left[\frac{\partial^2
		S_T^{(\ell)}(\theta_0)}{\partial \theta\ \partial
		\theta'}\right]^{-1}\frac{\partial  S_T^{(\ell)}(\theta_0)}{\partial
		\theta}+{o}\big(\hat{\theta}_T^{(\ell)}-\theta_0\big).
	\end{equation*}
According to the mentioned properties of the function $S_T^{(\ell)}(\theta)$, it can be differentiated under the integral sign, e.g.,
	\begin{equation}\label{PartDeriv1}
	\frac{\partial S_T^{(\ell)}(\theta)}{\partial \theta} =
	2 \idotsint\limits_{\R^\ell}g(\mathbf{u})\left[ \varphi_X^{(\ell)}(\mathbf{u};\theta)-\widetilde{\varphi}_T(\mathbf{u})\right] \frac{\partial \varphi_X^{(p)}(\mathbf{u};\theta)}{\partial \theta}\mathbf{du},
	\end{equation}
	and
	\begin{eqnarray}\label{PartDeriv2}
\nonumber\frac{\partial^2 S_T^{(\ell)}(\theta)}{\partial \theta\ \partial
		\theta'}&=&
	2\idotsint\limits_{\R^\ell}g(\mathbf{u})\left\{ \frac{\partial \varphi_X^{(\ell)}(\mathbf{u};\theta)}{\partial \theta}
	\frac{\partial \varphi_X^{(\ell)}(\mathbf{u};\theta)}{\partial \theta'}\right.\\
	&& \left.+ \left[\varphi_X^{(\ell)}(\mathbf{u};\theta)- \widetilde{\varphi}_T(\mathbf{u})\right]\frac{\partial^2 \varphi_X^{(p)}(\mathbf{u};\theta)}{\partial \theta\ \partial \theta'} \right\}	\mathbf{du}.
	\end{eqnarray}
 As $E\big[\widetilde{\varphi}_T^{(p)}(\mathbf{u})\big]=\varphi_Y^{(p)}(\mathbf{u};\theta_0)$, Eqs.(\ref{PartDeriv1})--(\ref{PartDeriv2}) give
	\begin{equation}\label{MathExpect}
	E\left[ \frac{\partial S_T^{(\ell)}(\theta_0)}{\partial \theta}\right]=0,
	\quad
	E\left[ \frac{\partial^2 S_T^{(\ell)}(\theta_0)}{\partial \theta\ \partial \theta'}\right]=2\mathbf{V},
	\end{equation}
	where
$$	\mathbf{V}=\idotsint_{\R^\ell}g(\mathbf{u})\frac{\partial \varphi_X^{(\ell)}(\mathbf{u};\theta_0)}{\partial \theta}
	\frac{\partial \varphi_X^{(\ell)}(\mathbf{u};\theta_0)}{\partial \theta'}\mathbf{du}.
$$
Under assumption {\rm(iii)}, the inequalities
$ 0 < \norm{\mathbf{V}} \leq\idotsint_{\R^\ell}g(\mathbf{u}) h(\mathbf{u})\mathbf{du}<+\infty $
	hold, and consequently
	\begin{equation}\label{A.S.Converg.}
	\left(\frac{\partial S_T^{(\ell)}(\theta_0)}{\partial \theta_0},
	\frac{\partial^2 S_T^{(\ell)}(\theta_0)}{\partial \theta\ \partial \theta'}\right)\stackrel{\rm as}\lrt (0,2\mathbf{V}),
	\quad T\rt +\infty.
	\end{equation}
		
Now, we write the gradient of $S_T^{(\ell)}(\theta)$ as
	\begin{equation*}
	\frac{\partial S_T^{(\ell)}(\theta)}{\partial \theta}
	=\frac2{T-\ell+1}\sum_{t=1}^{T-1} \mathbf{K}_t(\theta),
	\end{equation*}
	where
	$
	\mathbf{K}_t(\theta)=\idotsint_{\R^\ell}g(\mathbf{u})\left[\varphi_X^{(\ell)}(\mathbf{u};\theta)-\widetilde{\vphi}_T\left(\mathbf{u}\right)\right]
	\frac{\partial \varphi_X^{(\ell)}(\mathbf{u};\theta)}{\partial \theta}\,\mathbf{du}.
	$
	It can then be shown (see, for instance \citealp{Yu}) that the finite non--zero limit
	\begin{eqnarray*}
\mathbf{W}^2&:=& \lim\limits_{T\rt\infty} \frac1{(T-\ell+1)^2}\, \mathrm{Var}\Bigg[\sum_{t=1}^{T-\ell+1}\mathbf{K}_t(\theta_0)\Bigg]\\
&=&\lim\limits_{T\rt\infty} \frac1{(T-\ell+1)^2}\, \sum_{t=1}^{T-\ell+1}\sum_{s=1}^{T-\ell+1}\mbox{Cov}\Big[\mathbf{K}_t(\theta_0)\mathbf{K}_s(\theta_0)\Big]
\end{eqnarray*}
	exists if the series $\gamma_X(k):=\mbox{Cov}(X_t,X_{t+k})$,
	$k=0,\pm 1,\pm2,\dots$, has the finite and non--zero sum. In case of the Split-MA$(p)$ model, according to Eq.(\ref{Cov X}), we have
	\begin{eqnarray*}
		C&:=&\sum_{k=-\infty}^{+\infty}\gamma_X(k)=\sigma^2
		\Bigg(1+b_c\sum_{j=1}^{p}\alpha_j^2\Bigg)+2\sigma^2b_c\Bigg(\sum_{k=1}^{p-1}\sum_{j=1}^{p-k}\alpha_j\alpha_{j+k}-\sum_{k=1}^{p}\alpha_k \Bigg)\\
		&=& \sigma^2\Bigg[1+b_c\bigg(\sum_{j=1}^{p}\alpha_j\bigg)^2-2b_c
		\bigg(\sum_{j=1}^{p}\alpha_j\bigg)\Bigg]=\sigma^2(1-b_c).
	\end{eqnarray*}
According to this, it is obvious that for the non--trivial values $b_c\in(0,1)$, the inequalities $0<C<+\infty$ hold for each $\theta\in\Theta'$. By applying the central limit theorem for stationary processes, we obtain
	\[
	\sqrt{T-\ell+1}\;\frac{\partial S_T^{(\ell)}(\theta_0)}{\partial \theta}\stackrel{d}
	{\longrightarrow}\mathcal{N}(0,4\mathbf{W}^2),\quad T\rt+\infty,
	\]
	and this convergence and Eqs.(\ref{MathExpect})--(\ref{A.S.Converg.}) imply
	\[
	\sqrt{T-\ell+1} \left(\hat{\theta}_T^{(\ell)}-\theta_0\right)\stackrel{d}
	{\longrightarrow}\mathcal{N}(0,\mathbf{V}^{-1}\mathbf{W}^2\mathbf{V}^{-1}),\quad T\rt+\infty.
	\]
This completes the proof of the Theorem.	
\end{proof}

\begin{remark}
The assumption (i) of the previous theorem ensures the compactness of the parameters set, and it is necessary to provide the strong consistency of the ECF estimates of Split--MA model's parameters. On the other hand, assumptions (ii) and (iii) enable that regular random matrices $\partial^2 S_T^{(\ell)}(\theta_0)/\partial \theta\ \partial \theta'$ converge almost surely to the non-zero finite matrix $2\mathbf{V}$, when $T\rt\infty$. In this way, these two assumptions are necessary for AN of the ECF estimates. Note that all these assumptions are weaker than the general regularity conditions of ECF estimators (see, for instance \citealp{Knight&Yu}). Hence, they can be satisfied in most applications of the GSB process.
\end{remark}

\begin{remark}
According to the proof of the previous theorem, continuous function $S_0^{(\ell)}(\theta)\geq 0$, providing that $\ell\geq p+1$, attains the unique minimum at $\theta=\theta_0$. This is the so-called identification condition, which holds if the order   of the CF of Split--MA model is at least equal to the number of its parameters.
As we mentioned before, we consider the  simplest model of the Split--MA process, when $p=\alpha_1=1$ and $\theta=(b_c,\sigma^2)'$.
Hence,  we use the CF of order $\ell=p+1=2$, i.e., the ECF procedure based on two--dimensional random vector $\X_t^{(2)}:=(X_t,X_{t+1})'$.
Note that, in this case, the objective function
$S_T^{(2)}$ represents a double integral with respect to the weight function $g:\mathbb R^2\rightarrow\mathbb R^+$. Therefore, as we will see later on, it can be numerically approximated by using some cubature formulas. In addition, we give an explicit expression of the two--dimensional CF of $(X_t)$ in Eq.(\ref{Eq.3.6}).
As it can be seen, this function takes only real values, and the following real--valued function
\[
\widetilde{\varphi}_T(\mathbf{u}) :=\mathrm{Re}\ \widetilde{\varphi}_T^{(2)}(\mathbf{u})=\frac{1}{T-1}\sum_{t=1}^{T-1} \cos\left(u_1X_{t}+u_2X_{t+1} \right)
\]
may be used as its empirical estimate.
Figure \ref{Fig.4} illustrates the graphs of two--dimensional CF $\vphi_X^{(2)}(\mathbf{u})$, as well as the appropriate ECF $\widetilde{\varphi}_T(\mathbf{u})$, when $T=1\,500$ and $p=\alpha_1=c=\sigma^2=1$.
\end{remark}

\begin{figure}[hbtp]
	\includegraphics[width=.95\textwidth]{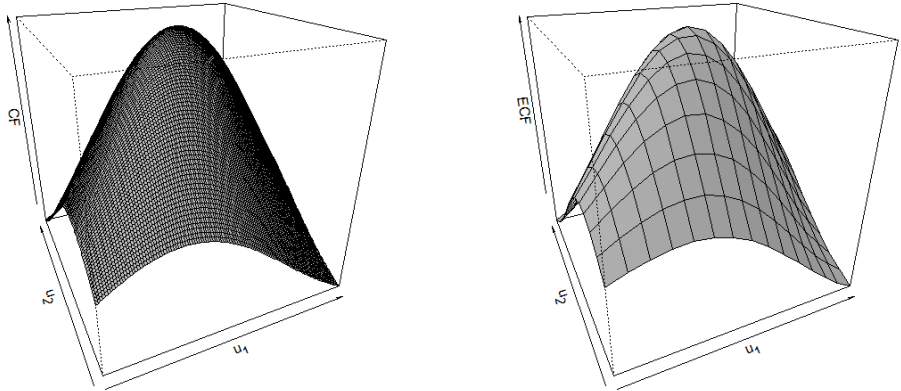}
	\caption{Graphs of the two--dimensional CF (panel left) and the appropriate ECF (panel right) of the series $\X_t^{(2)}=(X_t,X_{t+1})$.}\label{Fig.4}
\end{figure}

\section{Numerical simulations of the ECF estimates}
\label{sec:5}

\subsection{Simulating the Split--MA process}

In the first part of this Section, we present the pseudo algorithm intended to simulate the Split--MA(1) model, as well as to compute its initial parameters' estimates. It is based on 1\,000 independent Monte Carlo replications of our model, i.e., the 1\,000 independent realizations of the series
\[
X_t=\eps_t-\theta_{t-1}\, \eps_{t-1},\quad t=1,\dots,T.
\]
where $\theta_t=I(\eps_{t-1}^2\leq c)$ and $\eps_0=\eps_{-1}\stackrel{as}=0$. We consider two different simple sizes, $T=1\,500$ (large sample) and $T=150$ (small sample). This is primarily to show that convergence
\[
\frac{1}{T}\sum_{t=1}^{T}\theta_t\stackrel{as}\lrt b_c,\quad T\rt\infty,
\]
implies that the procedure of parameters estimation of the Split--MA(1) model can be applied in the case of short time series (see \citealp[pp. 57]{Split-BREAK}). Naturally, the number of observations ($T$) implies the different values of the appropriate estimation errors, which will be also investigated.

In the first estimation step, we have computed the estimates obtained by the method of moments. The autocorrelation function of the Split-MA(1) process $(X_t)$ is given by
\begin{equation}\label{Eq.5.2}
\rho(h)=\mathrm{Corr}(X_{t+h}, X_t)=\left\{%
\begin{array}{ll}
1, & \quad h=0, \\\
-b_c/(b_c+1), &\quad h=\pm 1, \\\
0, & \quad \hbox{otherwise}, \\
\end{array}%
\right.
\end{equation}
and as the estimates of $b_c$ we used the statistic
$ \widetilde{b}_c=-{\hat{\rho}_{_T}(1)}(1+\hat{\rho}_{_T}(1))^{-1},$
where $\hat{\rho}_{_T}(1)$ is the empirical first correlation
of $(X_t)$. Note that the inequalities $0<\widetilde{b}_c<1$ hold if and only if $-0,5< \hat{\rho}_{_T}(1) < 0$.
After that, according to Eq.(\ref{Cov X}), the estimate of the $\sigma^2$ can be computed as $\widetilde{\sigma}^2=\hat{\gamma}_X(0)(1+\widetilde{b}_c)^{-1}$, where $\hat{\gamma}_X(0)$ is the empirical variance of $(X_t)$. Finally, by solving the equation
$P\big\{\eps_t^2\leq
c\big\}=\widetilde{b}_c$ with respect to $c$, we obtain the
estimate of the critical value $\widetilde{c}={\widetilde{\sigma}^2} F^{-1}({\widetilde{b}_c})$, where $F(x)$ is the CDF of $\chi_1^2$ distributed random variable.
Using some well-known facts about the continuity of the stochastic convergences (cf. \citealp[pp. 24,118]{Serfling}), the strong consistency and asymptotic normality of the mentioned estimates can be proven (see, for more details \citealp{Split-BREAK}). In the following part, we describe in detail the ECF procedure for parameters estimation of the Split--MA model, and we also compare its efficiency with the estimates obtained by the method of moment.

\subsection{Computing the ECF--based estimator}
As we mentioned above, the ECF estimation procedure of the  parameters of Split--MA(1) process is based on  the minimization of the following double integral
\begin{equation}\label{Eq.5.1}
S_T^{(2)}(\theta)=\iint_{\R^2} g(u_1,u_2)\abs{\varphi_X^{(2)}(u_1,u_2;\theta)-\widetilde{\varphi}_T
	(u_1,u_2)}^2 \D u_1 \D u_2,
\end{equation}
with respect to the weight function $g:\R^2\to \R^+$. In our investigation, some typical exponential weight functions are considered. These functions put more weights around the origin, which is in accordance with the fact that CF in this point contains the most of information about the probability distribution of estimated model. On the other hand, exponential weights have a numerical advantage, because the integral in (\ref{Eq.5.1}) can be numerically approximated by using some $N$-point cubature formula
\begin{equation}\label{Cubat}
I(f;g):=\iint_{\R^2} g(u_1,u_2)f(u_1,u_2)du_1du_2\approx C_N(f):=\sum_{j=1}^{N}\omega_j f(u_{1j},u_{2j}),
\end{equation}
where $(u_{1j},u_{2j})\in\R^2$ are the cubature nodes and $\omega_j$ are the corresponding weight coefficients.

A particular problem here is a choice of the weight function $g(u_1,u_2)$. For this purpose, we consider the weight functions $g_k(u_1,u_2)=\exp(-\frac{k}{2}(u_1^2+u_2^2))$,  where $k\in\{1,2,3\}$.
For all of these weights we use a product cubature formula based on the one-dimensional  Gauss-Radau formula (cf. \citealp[pp. 329--330]{Mas_Mil2008} or \citealp{GVM2015}) with respect to an exponential weight on $(0,+\infty)$. Namely, introducing the polar coordinates $u_1=r\cos\theta$ and
$u_2=r\sin\theta$, the integral $I(f;g)$ in (\ref{Cubat}) is reduced to
\begin{equation}\label{GRF}
I(f;g)=\int_0^{+\infty} r \E^{-\gamma r^2}S(r)\,{\D}r,
\end{equation}
where  $S(r)$ is given by
\begin{equation}\label{TRF}
S(r)=\int_{\pi}^{\pi}f(r\cos\theta,r\sin\theta)\,{\D}\theta.
\end{equation}
The integral (\ref{TRF}) can be approximated by the composite trapezoidal rule in  $4m$ points $\theta_j=-\pi+j\pi/(2m)$, $j=0,1,\ldots,4m$, as
\[S(r)\approx S_m(r)=\frac{2\pi}{4m}\left\{\frac12f(-r,0)+\sum_{j=1}^{4m-1} f(r\cos\theta_j,r\sin\theta_j)+\frac12f(-r,0)\right\}.\]
Using the nodes
$x_\nu=\cos({\nu\pi}/{2m}),\; y_\nu=\sin({\nu\pi}/{2m}),\;
\nu=1,\ldots,m,$
after certain transformations, $S_m(r)$ can be represented in the form
\[S_m(r)=\frac{\pi}{2m}\sum_{\nu=1}^m\Bigl[f(r x_\nu,ry_\nu)+f(-r x_\nu,-ry_\nu)+ f(r y_\nu,-rx_\nu)+f(-r y_\nu,rx_\nu)\Bigr].\]
In this way, we obtain the cubature formula
\begin{eqnarray*}
C_N(f) &=& 2\pi A_0f(0,0)+\frac{\pi}{2m}\sum_{k=1}^nA_k\sum_{\nu=1}^m\bigl[f(r_k x_\nu,r_ky_\nu)+f(-r_k x_\nu,-r_ky_\nu)\\
&&+ f(r_k y_\nu,-r_kx_\nu)+f(-r_k y_\nu,r_kx_\nu)\bigr],
\end{eqnarray*}
where $A_k$ are weights of the one--dimensional $(n+1)$-point Gauss--Radau formula
\begin{equation*}\label{G-Radau}
\int_0^{+\infty} r \E^{-\gamma r^\alpha}S(r)\,{\D}r\approx A_0\, S(0)+\sum_{k=1}^nA_k\,S(r_k).
\end{equation*}
On the other hand, the nodes $r_k$ are zeros of the polynomial $\pi_n(r)$ orthogonal on $(0,+\infty)$ with respect to the exponential weight function  $r\mapsto r^2 \E^{-\gamma r^2}$. In our calculations we use $C_N$ with $N=81$ nodes ($n=5$, $m=4$). The numerical construction of the Gauss-Radau formulas can be done by the {\sc Mathematica} package {\tt  ``OrthogonalPolynomials''}, for an arbitrary number of points (see, for more details \citealp{GVM1,GVM2}).

After numerical construction of cubature rules, the objective function (\ref{Eq.5.1}) is minimized by a Nelder-Mead method, and the estimation procedure is realized by the original authors' codes written in statistical programming language ``R". The estimates obtained by the method of moments were used as the initial estimated values, and the performance of these and ECF estimates was examined in case of the simplest, Split-MA$(1)$  model. For a true value of the parameter it was chosen the vector $\theta_0=(b_c, \sigma^2)=(0.6827,1)$, where $c=\sigma^2F^{-1}(b_c)=1$.

\subsection{Simulations Results}
We apply the ECF method, with the initial values that were obtained by the previously described method of moments estimation procedure. In this way, we compute the ECF estimates of parameters $b_c, c,\sigma^2$ of the Split--MA$(1)$ model. Their summarized values, i.e. the true parameter values (TRUE), averages of estimated parameters (MEAN), together with their minimums (MIN), maximums (MAX), bias (BIAS) and the corresponding root mean squared errors (RMSE), are set in rows of the Tables \ref{Tab:1} and \ref{Tab:2}. In the first (numerical) column of both Tables there are the estimated values of the initial estimates, obtained by the method of moments. The following three columns contain the estimated parameters' values obtained by the ECF procedures, with respect to the weights $g_k(u_1,u_2)$, $k=1,2,3$. As it can be seen, in comparison to initial estimates, the ECF estimates have a smaller estimation errors.  Also, the averages of the ECF estimates of the all estimated parameters  are close to their true values.

\begin{table}[htbp]{\small
		\caption{Summarized values of the parameters' estimators of Split--MA$(1)$ process, obtained by Monte Carlo study of the model.}\label{Tab:1}
		\begin{tabular}{clcccc}\hline
			\multicolumn{2}{c}{\raisebox{-3.5mm}[0pt]{Parameters}}
			& \raisebox{-2.5mm}[0pt]{Initial} & \multicolumn{3}{c}{\raisebox{-2.5mm}[0pt]{ECF estimates/weights}}
			\vspace{1.5mm}\tabularnewline
			\cline{4-6}
			& & \raisebox{1.5mm}[0pt]{\ estimates\ } & \raisebox{-2mm}[0pt]{$g_1(u_1,u_2)$} & \raisebox{-2mm}[0pt]{$g_2(u_1,u_2)$} & \raisebox{-2mm}[0pt]{$g_3(u_1,u_2)$}
			\vspace{1.2mm}\tabularnewline
			\hline
			&\raisebox{-1mm}[0pt]{TRUE} &\raisebox{-1mm}[0pt]{0.6827}
			&\raisebox{-1mm}[0pt]{0.6827}
			&\raisebox{-1mm}[0pt]{0.6827} &\raisebox{-1mm}[0pt]{0.6827}
			\tabularnewline	
			& \raisebox{-1mm}[0pt]{MIN}
			& \raisebox{-1mm}[0pt]{0.2661}
			&\raisebox{-1mm}[0pt]{0.4093} &\raisebox{-1mm}[0pt]{0.3999} &\raisebox{-1mm}[0pt]{0.5131}
		
			\tabularnewline
			&  \raisebox{-1mm}[0pt]{MEAN} &\raisebox{-1mm}[0pt]{0.6781}
			&\raisebox{-1mm}[0pt]{0.6867}
			&\raisebox{-1mm}[0pt]{0.6865} &\raisebox{-1mm}[0pt]{0.6859}
			
			\tabularnewline
\raisebox{1.5mm}[0pt]{${b_c}$}
			&  \raisebox{-1mm}[0pt]{MAX}  &  \raisebox{-1mm}[0pt]{0.9962} &
			\raisebox{-1mm}[0pt]{0.9065} & \raisebox{-1mm}[0pt]{0.9967} & \raisebox{-1mm}[0pt]{0.8959}
			
			\tabularnewline
			&  \raisebox{-1mm}[0pt]{BIAS}  &  \raisebox{-1mm}[0pt]{-4.55E-03} &
			\raisebox{-1mm}[0pt]{4.04E-03} & \raisebox{-1mm}[0pt]{3.81E-03}
			& \raisebox{-1mm}[0pt]{3.23E-03}
			
			\tabularnewline
			&  \raisebox{-1mm}[0pt]{RMSE}  &  \raisebox{-1mm}[0pt]{0.1452} &
			\raisebox{-1mm}[0pt]{0.0731} & \raisebox{-1mm}[0pt]{0.0888} & \raisebox{-1mm}[0pt]{0.0641}
				
			\vspace{1mm}\tabularnewline
			\hline
			
				&  \raisebox{-1mm}[0pt]{TRUE} &\raisebox{-1mm}[0pt]{1.0000}
				&\raisebox{-1mm}[0pt]{1.0000}
				&\raisebox{-1mm}[0pt]{1.0000} &\raisebox{-1mm}[0pt]{1.0000}
				\tabularnewline	
				& \raisebox{-1mm}[0pt]{MIN}
				& \raisebox{-1mm}[0pt]{0.3957}
				&\raisebox{-1mm}[0pt]{0.6078} &\raisebox{-1mm}[0pt]{0.5696} &\raisebox{-1mm}[0pt]{0.5711}
				
				\tabularnewline
				&  \raisebox{-1mm}[0pt]{MEAN} &\raisebox{-1mm}[0pt]{1.0477}
				&\raisebox{-1mm}[0pt]{1.0256}
				&\raisebox{-1mm}[0pt]{1.0282} &\raisebox{-1mm}[0pt]{1.0226}
				
				\tabularnewline
				\raisebox{1.5mm}[0pt]{$c$}
				&  \raisebox{-1mm}[0pt]{MAX} &  \raisebox{-1mm}[0pt]{1.9911} &
				\raisebox{-1mm}[0pt]{1.5952} & \raisebox{-1mm}[0pt]{1.6843} & \raisebox{-1mm}[0pt]{1.6893}
				
				\tabularnewline
				&  \raisebox{-1mm}[0pt]{BIAS}  &  \raisebox{-1mm}[0pt]{4.77E-02} &
				\raisebox{-1mm}[0pt]{2.56-02} & \raisebox{-1mm}[0pt]{2.82E-02} & \raisebox{-1mm}[0pt]{2.26E-02}
				
				\tabularnewline
				&  \raisebox{-1mm}[0pt]{RMSE}  &  \raisebox{-1mm}[0pt]{0.2880} &
				\raisebox{-1mm}[0pt]{0.1464} & \raisebox{-1mm}[0pt]{0.2291} & \raisebox{-1mm}[0pt]{0.1740}
				
				\vspace{1mm}\tabularnewline
				\hline
		
			&  \raisebox{-1mm}[0pt]{TRUE} &\raisebox{-1mm}[0pt]{1.0000}
			&\raisebox{-1mm}[0pt]{1.0000}
			&\raisebox{-1mm}[0pt]{1.0000} &\raisebox{-1mm}[0pt]{1.0000}
			\tabularnewline	
			& \raisebox{-1mm}[0pt]{MIN}
			& \raisebox{-1mm}[0pt]{0.6534}
			&\raisebox{-1mm}[0pt]{0.7822} &\raisebox{-1mm}[0pt]{0.7641} &\raisebox{-1mm}[0pt]{0.8292}
			
			\tabularnewline
			&  \raisebox{-1mm}[0pt]{MEAN} &\raisebox{-1mm}[0pt]{0.9880}
			&\raisebox{-1mm}[0pt]{0.9984}
			&\raisebox{-1mm}[0pt]{0.9963} &\raisebox{-1mm}[0pt]{0.9984}
			
			\tabularnewline
			\raisebox{1.5mm}[0pt]{$\sigma^2$}
			&  \raisebox{-1mm}[0pt]{MAX}  &  \raisebox{-1mm}[0pt]{1.5151} &
			\raisebox{-1mm}[0pt]{1.2311} & \raisebox{-1mm}[0pt]{1.2974} & \raisebox{-1mm}[0pt]{1.2128}
			
			\tabularnewline
			&  \raisebox{-1mm}[0pt]{BIAS}  &  \raisebox{-1mm}[0pt]{-1.20E-02} &
			\raisebox{-1mm}[0pt]{-1.60E-03} & \raisebox{-1mm}[0pt]{-3.70E-03} & \raisebox{-1mm}[0pt]{-1.63E-03}
			
			\tabularnewline
			&  \raisebox{-1mm}[0pt]{RMSE}  &  \raisebox{-1mm}[0pt]{0.1482} &
			\raisebox{-1mm}[0pt]{0.0680} & \raisebox{-1mm}[0pt]{0.0851} & \raisebox{-1mm}[0pt]{0.0597}
			
			\vspace{1mm}\tabularnewline
			\hline
			
			\raisebox{-.5mm}[0pt]{\small Sample size:} &	\raisebox{-.5mm}[0pt]{\small $T=150$} \\  \raisebox{-.5mm}[0pt]{\small Weights:\quad\;\;} & \multicolumn{4}{l}{\raisebox{-.5mm}[0pt]{$g_k(u_1,u_2)=\exp(-k(u_1^2+u_2^2)/2)$,\, $k\in\{1,2,3\}\;$}}
		\end{tabular}}
	\end{table}
	
	\begin{table}[htbp]{\small
			\caption{Summarized values of the parameters' estimators of Split--MA$(1)$ process, obtained by Monte Carlo study of the model.}\label{Tab:2}
			\begin{tabular}{clcccc}\hline
				\multicolumn{2}{c}{\raisebox{-3.5mm}[0pt]{Parameters}}
				& \raisebox{-2.5mm}[0pt]{Initial} & \multicolumn{3}{c}{\raisebox{-2.5mm}[0pt]{ECF estimates/weights}}
				\vspace{1.5mm}\tabularnewline
				\cline{4-6}
				& & \raisebox{1.5mm}[0pt]{\ estimates\ } & \raisebox{-2mm}[0pt]{$g_1(u_1,u_2)$} & \raisebox{-2mm}[0pt]{$g_2(u_1,u_2)$} & \raisebox{-2mm}[0pt]{$g_3(u_1,u_2)$}
				\vspace{1.2mm}\tabularnewline
				\hline
				&  \raisebox{-1mm}[0pt]{TRUE} &\raisebox{-1mm}[0pt]{0.6827}
				&\raisebox{-1mm}[0pt]{0.6827}
				&\raisebox{-1mm}[0pt]{0.6827} &\raisebox{-1mm}[0pt]{0.6827}
				\tabularnewline	
				& \raisebox{-1mm}[0pt]{MIN}
				& \raisebox{-1mm}[0pt]{0.5212}
				&\raisebox{-1mm}[0pt]{0.5773} &\raisebox{-1mm}[0pt]{0.5494} &\raisebox{-1mm}[0pt]{0.5442}
				
				\tabularnewline
				&  \raisebox{-1mm}[0pt]{MEAN} &\raisebox{-1mm}[0pt]{0.6851}
				&\raisebox{-1mm}[0pt]{0.6833}
				&\raisebox{-1mm}[0pt]{0.6841} &\raisebox{-1mm}[0pt]{0.6830}
				
				\tabularnewline
				\raisebox{1.5mm}[0pt]{${b_c}$}
				&  \raisebox{-1mm}[0pt]{MAX}  &  \raisebox{-1mm}[0pt]{0.8686} &
				\raisebox{-1mm}[0pt]{0.8169} & \raisebox{-1mm}[0pt]{0.8304} & \raisebox{-1mm}[0pt]{0.8028}
				
				\tabularnewline
				&  \raisebox{-1mm}[0pt]{BIAS}  &  \raisebox{-1mm}[0pt]{2.44E-03} &
				\raisebox{-1mm}[0pt]{6.39E-04} & \raisebox{-1mm}[0pt]{1.42E-03} & \raisebox{-1mm}[0pt]{2.64E-04}
				
				\tabularnewline
				&  \raisebox{-1mm}[0pt]{RMSE}  &  \raisebox{-1mm}[0pt]{0.0532} &
				\raisebox{-1mm}[0pt]{0.0406} & \raisebox{-1mm}[0pt]{0.0520} & \raisebox{-1mm}[0pt]{0.0417}
				
				\vspace{1mm}\tabularnewline
				\hline
				
				&  \raisebox{-1mm}[0pt]{TRUE} &\raisebox{-1mm}[0pt]{1.0000}
				&\raisebox{-1mm}[0pt]{1.0000}
				&\raisebox{-1mm}[0pt]{1.0000} &\raisebox{-1mm}[0pt]{1.0000}
				\tabularnewline	
				& \raisebox{-1mm}[0pt]{MIN}
				& \raisebox{-1mm}[0pt]{0.5859}
				&\raisebox{-1mm}[0pt]{0.6532} &\raisebox{-1mm}[0pt]{0.7176} &\raisebox{-1mm}[0pt]{0.6876}
				
				\tabularnewline
				&  \raisebox{-1mm}[0pt]{MEAN} &\raisebox{-1mm}[0pt]{1.0190}
				&\raisebox{-1mm}[0pt]{1.0117}
				&\raisebox{-1mm}[0pt]{1.0125} &\raisebox{-1mm}[0pt]{0.9916}
				
				\tabularnewline
				\raisebox{1.5mm}[0pt]{$c$}
				&  \raisebox{-1mm}[0pt]{MAX} &  \raisebox{-1mm}[0pt]{1.5952} &
				\raisebox{-1mm}[0pt]{1.4944} & \raisebox{-1mm}[0pt]{1.4601} & \raisebox{-1mm}[0pt]{1.4736}
				
				\tabularnewline
				&  \raisebox{-1mm}[0pt]{BIAS}  &  \raisebox{-1mm}[0pt]{1.90E-02} &
				\raisebox{-1mm}[0pt]{1.17E-02} & \raisebox{-1mm}[0pt]{1.25E-02} & \raisebox{-1mm}[0pt]{-7.42E-03}
				
				\tabularnewline
				&  \raisebox{-1mm}[0pt]{RMSE}  &  \raisebox{-1mm}[0pt]{0.1886} &
				\raisebox{-1mm}[0pt]{0.1527} & \raisebox{-1mm}[0pt]{0.1329} & \raisebox{-1mm}[0pt]{0.0847}
				
				\vspace{1mm}\tabularnewline
				\hline
				
				&  \raisebox{-1mm}[0pt]{TRUE} &\raisebox{-1mm}[0pt]{1.0000}
				&\raisebox{-1mm}[0pt]{1.0000}
				&\raisebox{-1mm}[0pt]{1.0000} &\raisebox{-1mm}[0pt]{1.0000}
				\tabularnewline	
				& \raisebox{-1mm}[0pt]{MIN}
				& \raisebox{-1mm}[0pt]{0.8375}
				&\raisebox{-1mm}[0pt]{0.8883} &\raisebox{-1mm}[0pt]{0.8436} &\raisebox{-1mm}[0pt]{0.8836}
				
				\tabularnewline
				&  \raisebox{-1mm}[0pt]{MEAN} &\raisebox{-1mm}[0pt]{1.0012}
				&\raisebox{-1mm}[0pt]{0.9992}
				&\raisebox{-1mm}[0pt]{0.9988} &\raisebox{-1mm}[0pt]{0.9996}
				
				\tabularnewline
				\raisebox{1.5mm}[0pt]{$\sigma^2$}
				&  \raisebox{-1mm}[0pt]{MAX}  &  \raisebox{-1mm}[0pt]{1.1730} &
				\raisebox{-1mm}[0pt]{1.1131} & \raisebox{-1mm}[0pt]{1.1518} & \raisebox{-1mm}[0pt]{1.1194}
				
				\tabularnewline
				&  \raisebox{-1mm}[0pt]{BIAS}  &  \raisebox{-1mm}[0pt]{1.21E-03} &
				\raisebox{-1mm}[0pt]{-7.91E-04} & \raisebox{-1mm}[0pt]{-1.20E-03} & \raisebox{-1mm}[0pt]{-4.04E-04}
				
				\tabularnewline
				&  \raisebox{-1mm}[0pt]{RMSE}  &  \raisebox{-1mm}[0pt]{0..0507} &
				\raisebox{-1mm}[0pt]{0.0374} & \raisebox{-1mm}[0pt]{0.0489} & \raisebox{-1mm}[0pt]{0.0380}
				
				\vspace{1mm}\tabularnewline
				\hline
				
			\raisebox{-.5mm}[0pt]{\small Sample size:} &	\raisebox{-.5mm}[0pt]{\small $T=1\,500$} \\  \raisebox{-.5mm}[0pt]{\small Weights:\quad\;\;} & \multicolumn{4}{l}{\raisebox{-.5mm}[0pt]{$g_k(u_1,u_2)=\exp(-k(u_1^2+u_2^2)/2)$,\, $k\in\{1,2,3\}\;$}}
			\end{tabular}}
		\end{table}

	What follows is an empirical investigation of the asymptotic properties (strong consistency and asymptotic normality) of the parameters' estimators of our model, which were formally proved in Theorem \ref{Thm:4}. One should remark that in our simulation study, for all of the observed weights, the ECF estimates have a more prominent stability, compared to the appropriate initial estimates. This can be easily seen according to their estimated errors, as well as in Figures \ref{Fig.5} and \ref{Fig.6}: the initial estimates (panels above), and the ECF estimates with the weights $g_k(u_1,u_2)$, $k=1,2,3$ (panels bellow).
	
	Furthermore, we have obtained some testing results concerning the AN of initial and ECF estimates of our model, which is expected according to Theorem \ref{Thm:4}. For this purpose, we have used Anderson-Darling and Cramer-von Mises tests of normality, which test statistics (labeled as $AD$ and $W$, respectively), as well as the corresponding $p$--values, are computed by using the appropriate procedures from the R-package ``nortest'', authorized by \citet{Gross}. Additionally, we have used the composite Jarque-Bera (JB) test of normality, with specified number of the Monte-Carlo replications, realized in R-package ``normtest" (\citealp{Normtest}). These values are shown in Tables \ref{Tab:3} and \ref{Tab:4}, where one can see that the AN is mostly confirmed for the ECF estimates of parameters $b_c$ and $\sigma^2$.
	On the other hand, it varies to a some degree in the case of parameter $c$, as a consequence of the multi--stage estimation procedure of this model's parameters.
	Certain confirmations of these facts are also given in the histograms of their empirical distribution, presented in Figures \ref{Fig.5} and \ref{Fig.6}.
	
	\begin{figure}[htbp]
		\begin{center}
			\includegraphics[width=.95\textwidth]{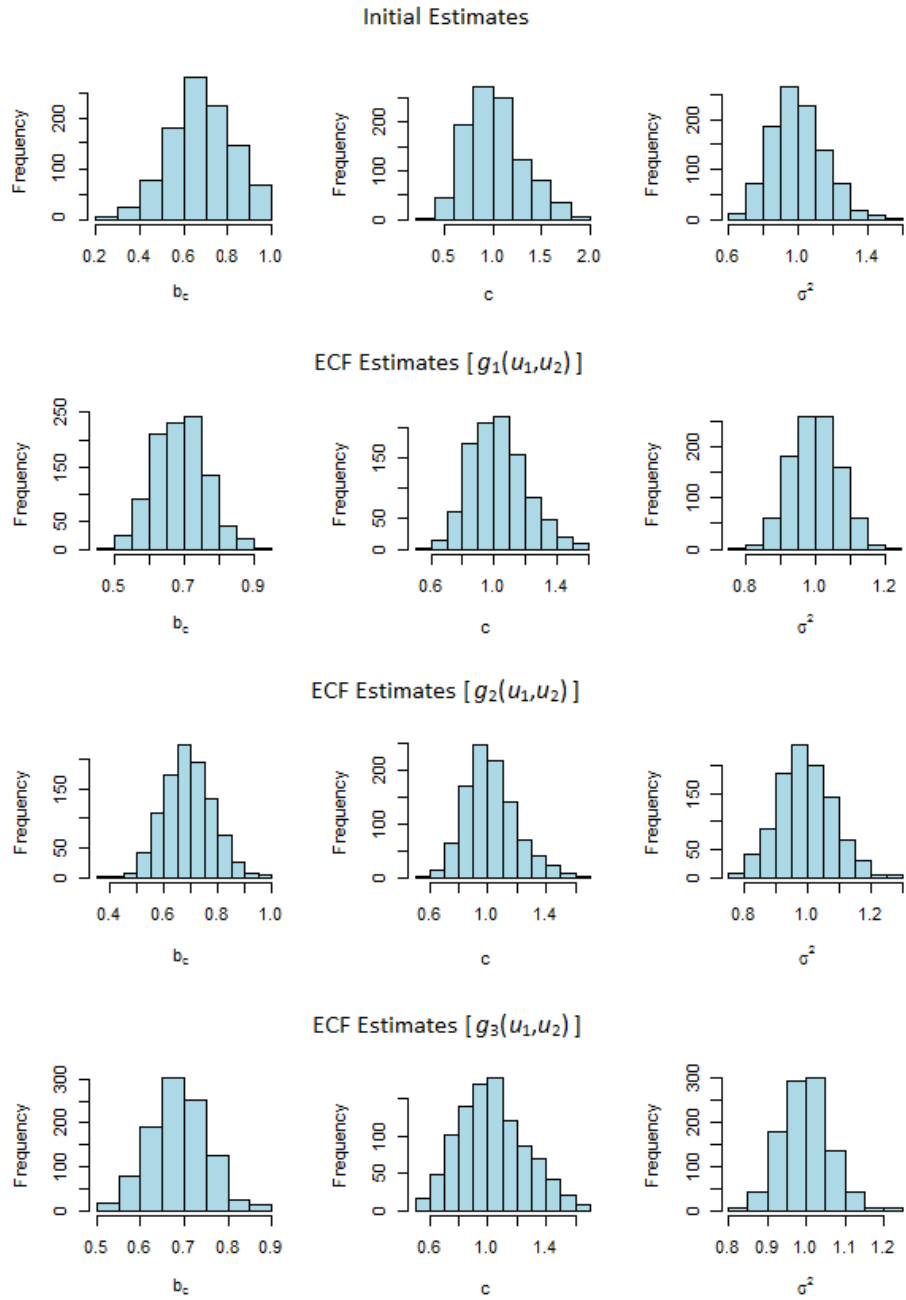}
			\caption{Empirical distributions of estimated parameters. (Sample size: $T=150$.)}\label{Fig.5}
		\end{center}
	\end{figure}
	
	\begin{figure}[htbp]
		\begin{center}
			\includegraphics[width=.95\textwidth]{Figure6.eps}
			\caption{Empirical distributions of estimated parameters. (Sample size: $T=1\,500$.)}\label{Fig.6}
		\end{center}
	\end{figure}

	\begin{table}[htbp]{\small
			\caption{Summarized results of normality test statistics (Sample size: $T=150$.)}\label{Tab:3}
			\begin{tabular}{cccccc}\hline
				\raisebox{-4.5mm}[0pt]{Parameters} &	\raisebox{-4.5mm}[0pt]{Statistics}
				& \raisebox{-3.5mm}[0pt]{Initial} & \multicolumn{3}{c}{\raisebox{-3.5mm}[0pt]{ECF estimates/weights}}
				\vspace{1.5mm}\tabularnewline
				\cline{4-6}
				& & \raisebox{1.5mm}[0pt]{\ estimates\ } & \raisebox{-2mm}[0pt]{$g_1(u_1,u_2)$} & \raisebox{-2mm}[0pt]{$g_2(u_1,u_2)$} & \raisebox{-2mm}[0pt]{$g_3(u_1,u_2)$}
				\vspace{1.2mm}\tabularnewline
				\hline
				& \raisebox{-.25mm}[0pt]{$AD$} &
				\raisebox{-.25mm}[0pt]{0.8398$^*$} &
				\raisebox{-.25mm}[0pt]{0.6280\;} &
				\raisebox{-.25mm}[0pt]{0.3471} &
				\raisebox{-.25mm}[0pt]{0.4576\;}
				\tabularnewline
				&  \raisebox{.25mm}[0pt]{($p$--value)} &\raisebox{.25mm}[0pt]{(0.0305)}
				& \raisebox{.25mm}[0pt]{(0.1016)}
				&\raisebox{.25mm}[0pt]{(0.4794)}
				&\raisebox{.25mm}[0pt]{(0.2642)}
				\tabularnewline
			 &	\raisebox{-.5mm}[0pt]{$W$} &
				\raisebox{-.5mm}[0pt]{0.1266$^*$}
				& \raisebox{-.5mm}[0pt]{0.1043\,}
				& \raisebox{-.5mm}[0pt]{0.0822\;} &  \raisebox{-.5mm}[0pt]{0.0698\;}				
				\tabularnewline
				\raisebox{1.5mm}[0pt]{${b_c}$} &\raisebox{0mm}[0pt]{($p$--value)} &\raisebox{0mm}[0pt]{(0.0488)}
				&\raisebox{0mm}[0pt]{(0.0980)}
				&\raisebox{0mm}[0pt]{(0.1937)}
				&\raisebox{0mm}[0pt]{(0.2807)}
				\vspace{.2mm}\tabularnewline
			
&\raisebox{-.75mm}[0pt]{$JB$} &
\raisebox{-.75mm}[0pt]{7.7128$^*$}
& \raisebox{-.75mm}[0pt]{5.1060}
& \raisebox{-.75mm}[0pt]{3.6464} &  \raisebox{-.75mm}[0pt]{2.0886}					
\tabularnewline
&\raisebox{-.25mm}[0pt]{($p$--value)}
&\raisebox{-.25mm}[0pt]{\,(0.029)}
&\raisebox{-.25mm}[0pt]{\,(0.073)}
&\raisebox{-.25mm}[0pt]{\,(0.142)}
&\raisebox{-.25mm}[0pt]{\, (0.326)}
\vspace{1.2mm}\tabularnewline
	\hline
			& \raisebox{-.25mm}[0pt]{$AD$} &
			\raisebox{-.25mm}[0pt]{2.1285$^{**}$} &
			\raisebox{-.25mm}[0pt]{0.8923$^{*}$} &
			\raisebox{-.25mm}[0pt]{1.2947$^{**}$} &
			\raisebox{-.25mm}[0pt]{0.8813$^{*}$}
			\tabularnewline
			&  \raisebox{.25mm}[0pt]{($p$--value)} &\raisebox{.25mm}[0pt]{(2.07E-05)}
			& \raisebox{-.25mm}[0pt]{(0.0226)}
			&\raisebox{.25mm}[0pt]{(2.31E-03)}
			&\raisebox{.25mm}[0pt]{(0.0246)}
			\tabularnewline
				
			&	\raisebox{-.5mm}[0pt]{$W$} &
				\raisebox{-.5mm}[0pt]{0.3517$^{**}$}
			& \raisebox{-.5mm}[0pt]{0.1179}
			& \raisebox{-.5mm}[0pt]{0.1722$^{*}$} &  \raisebox{-.5mm}[0pt]{0.1640$^{*}$}				
			\tabularnewline
			\raisebox{1.5mm}[0pt]{${c}$}
			&\raisebox{0mm}[0pt]{($p$--value)} &\raisebox{0mm}[0pt]{(8.71E-05)}
			&\raisebox{0mm}[0pt]{(0.0639)}
			&\raisebox{0mm}[0pt]{(0.0122)}
			&\raisebox{0mm}[0pt]{(0.0156)}
			\vspace{.2mm}\tabularnewline
			
				&\raisebox{-.75mm}[0pt]{$JB$} &
				\raisebox{-.75mm}[0pt]{22.670$^{**}$}
				& \raisebox{-.75mm}[0pt]{7.6551$^{*}$}
				& \raisebox{-.75mm}[0pt]{13.350$^{**}$} &  \raisebox{-.75mm}[0pt]{7.0815$^{*}$}					
				\tabularnewline
			 &\raisebox{-.25mm}[0pt]{($p$--value)}
			 &\raisebox{-.25mm}[0pt]{(0.002)}
			&\raisebox{-.25mm}[0pt]{(0.021)}
			&\raisebox{-.25mm}[0pt]{(0.003)}
			&\raisebox{-.25mm}[0pt]{(0.036)}
			\vspace{1.2mm}\tabularnewline
		\hline
		
			& \raisebox{-.25mm}[0pt]{$AD$} &
			\raisebox{-.25mm}[0pt]{1.5831$^{**}$} &
			\raisebox{-.25mm}[0pt]{0.3479\;} &
			\raisebox{-.25mm}[0pt]{0.4613\;} &
			\raisebox{-.25mm}[0pt]{0.5573\;}
			\tabularnewline
			&  \raisebox{.25mm}[0pt]{($p$--value)} &\raisebox{.25mm}[0pt]{(1.02E-04)}
			& \raisebox{-.25mm}[0pt]{(0.4774)}
			&\raisebox{.25mm}[0pt]{(0.2587)}
			&\raisebox{.25mm}[0pt]{(0.1199)}
			\tabularnewline
			
			&	\raisebox{-.5mm}[0pt]{$W$} &
			\raisebox{-.5mm}[0pt]{0.2587$^{**}$}
			& \raisebox{-.5mm}[0pt]{0.0396\;}
			& \raisebox{-.5mm}[0pt]{0.0668\;} &  \raisebox{-.5mm}[0pt]{0.0623\;}				
			\tabularnewline
			\raisebox{1.5mm}[0pt]{$\sigma^{2}$}
			&\raisebox{0mm}[0pt]{($p$--value)} &\raisebox{0mm}[0pt]{(1.02E-04)}
			&\raisebox{0mm}[0pt]{(0.6898)}
			&\raisebox{0mm}[0pt]{(0.3080)}
			&\raisebox{0mm}[0pt]{(0.3523)}
			\vspace{.2mm}\tabularnewline
			
			&\raisebox{-.75mm}[0pt]{$JB$} &
			\raisebox{-.75mm}[0pt]{19.085$^{**}$}
			& \raisebox{-.75mm}[0pt]{0.6704\;}
			& \raisebox{-.75mm}[0pt]{2.1195\;} &  \raisebox{-.75mm}[0pt]{0.3181\;}					
			\tabularnewline
			&\raisebox{-.25mm}[0pt]{($p$--value)}
			&\raisebox{-.25mm}[0pt]{(2.E-03)}
			&\raisebox{-.25mm}[0pt]{(0.6860)}
			&\raisebox{-.25mm}[0pt]{(0.310)}
			&\raisebox{-.25mm}[0pt]{(0.836)}
			\vspace{1.2mm}\tabularnewline
		\hline
	\raisebox{-.5mm}[0pt]{\small $^{*}p<0.05$}  & \raisebox{-.5mm}[0pt]{\small$^{**}p<0.01$\qquad}
			\end{tabular}\vskip1cm

	\caption{Summarized results of normality test statistics. (Sample size: $T=1\,500$.)}\label{Tab:4}
			\begin{tabular}{cccccc}\hline
				\raisebox{-4.5mm}[0pt]{Parameters} &	\raisebox{-4.5mm}[0pt]{Statistics}
				& \raisebox{-3.5mm}[0pt]{Initial} & \multicolumn{3}{c}{\raisebox{-3.5mm}[0pt]{ECF estimates/weights}}
				\vspace{1.5mm}\tabularnewline
				\cline{4-6}
				& & \raisebox{1.5mm}[0pt]{\ estimates\ } & \raisebox{-2mm}[0pt]{$g_1(u_1,u_2)$} & \raisebox{-2mm}[0pt]{$g_2(u_1,u_2)$} & \raisebox{-2mm}[0pt]{$g_3(u_1,u_2)$}
				\vspace{1.2mm}\tabularnewline
				\hline
				
				& \raisebox{-.25mm}[0pt]{$AD$} &
				\raisebox{-.25mm}[0pt]{0.5039\;} &
				\raisebox{-.25mm}[0pt]{0.3845\;} &
				\raisebox{-.25mm}[0pt]{0.6432} &
				\raisebox{-.25mm}[0pt]{0.3802\;}
				\tabularnewline
				&  \raisebox{.25mm}[0pt]{($p$--value)} &\raisebox{.25mm}[0pt]{(0.2037)}
				& \raisebox{.25mm}[0pt]{(0.3943)}
				&\raisebox{.25mm}[0pt]{(0.0932)}
				&\raisebox{.25mm}[0pt]{(0.4025)}
				\tabularnewline
				&	\raisebox{-.5mm}[0pt]{$W$} &
				\raisebox{-.5mm}[0pt]{0.0513\;}
				& \raisebox{-.5mm}[0pt]{0.0664\;}
				& \raisebox{-.5mm}[0pt]{0.1138\;} &  \raisebox{-.5mm}[0pt]{0.0593\;}				
				\tabularnewline
				\raisebox{1.5mm}[0pt]{${b_c}$} &\raisebox{0mm}[0pt]{($p$--value)} &\raisebox{0mm}[0pt]{(0.4924)}
				&\raisebox{0mm}[0pt]{(0.3118)}
				&\raisebox{0mm}[0pt]{(0.0727)}
				&\raisebox{0mm}[0pt]{(0.3857)}
				\vspace{.2mm}\tabularnewline
				
	&\raisebox{-.75mm}[0pt]{$JB$} &
	\raisebox{-.75mm}[0pt]{7.6025$^*$}
	& \raisebox{-.75mm}[0pt]{0.5500\,}
	& \raisebox{-.75mm}[0pt]{2.7498\,} &  \raisebox{-.75mm}[0pt]{1.5539\; }					
	\tabularnewline
	&\raisebox{-.25mm}[0pt]{($p$--value)}
	&\raisebox{-.25mm}[0pt]{(0.031)}
	&\raisebox{-.25mm}[0pt]{(0.732)}
	&\raisebox{-.25mm}[0pt]{(0.248)}
	&\raisebox{-.25mm}[0pt]{(0.462)}
	\vspace{1.2mm}\tabularnewline
	\hline
	
	& \raisebox{-.25mm}[0pt]{$AD$} &
	\raisebox{-.25mm}[0pt]{0.9288$^{*}$} &
	\raisebox{-.25mm}[0pt]{0.6048} &
	\raisebox{-.25mm}[0pt]{1.0144$^{*}$} &
	\raisebox{-.25mm}[0pt]{0.7144}
	\tabularnewline
	&  \raisebox{.25mm}[0pt]{($p$--value)} &\raisebox{.25mm}[0pt]{(0.0184)}
	& \raisebox{-.25mm}[0pt]{(0.1159)}
	&\raisebox{.25mm}[0pt]{(0.0115)}
	&\raisebox{.25mm}[0pt]{(0.0622)}
	\tabularnewline
	
	&	\raisebox{-.5mm}[0pt]{$W$} &
	\raisebox{-.5mm}[0pt]{0.1517$^{*}$}
	& \raisebox{-.5mm}[0pt]{0.0898}
	& \raisebox{-.5mm}[0pt]{0.1644$^{*}$} &  \raisebox{-.5mm}[0pt]{0.1269$^{*}$}				
	\tabularnewline
	\raisebox{1.5mm}[0pt]{${c}$}
	&\raisebox{0mm}[0pt]{($p$--value)} &\raisebox{0mm}[0pt]{(0.0226)}
	&\raisebox{0mm}[0pt]{(0.1546)}
	&\raisebox{0mm}[0pt]{(0.0154)}
	&\raisebox{0mm}[0pt]{(0.0483)}
	\vspace{.2mm}\tabularnewline
	
	&\raisebox{-.75mm}[0pt]{$JB$} &
	\raisebox{-.75mm}[0pt]{6.6989$^{*}$}
	& \raisebox{-.75mm}[0pt]{2.1136\;}
	& \raisebox{-.75mm}[0pt]{6.1157$^{*}$} &  \raisebox{-.75mm}[0pt]{4.3429\; }					
	\tabularnewline
	&\raisebox{-.25mm}[0pt]{($p$--value)}
	&\raisebox{-.25mm}[0pt]{(0.042)}
	&\raisebox{-.25mm}[0pt]{(0.334)}
	&\raisebox{-.25mm}[0pt]{(0.048)}
	&\raisebox{-.25mm}[0pt]{(0.103)}
	\vspace{1.2mm}\tabularnewline
	\hline
				& \raisebox{-.25mm}[0pt]{$AD$} &
				\raisebox{-.25mm}[0pt]{0.3375\;} &
				\raisebox{-.25mm}[0pt]{0.3491\;} &
				\raisebox{-.25mm}[0pt]{0.6819\;} &
				\raisebox{-.25mm}[0pt]{0.3312\;}
				\tabularnewline
				&  \raisebox{.25mm}[0pt]{($p$--value)} &\raisebox{.25mm}[0pt]{(0.5039)}
				& \raisebox{-.25mm}[0pt]{(0.4745)}
				&\raisebox{.25mm}[0pt]{(0.0748)}
				&\raisebox{.25mm}[0pt]{(0.5122)}
				\tabularnewline
				
				&	\raisebox{-.5mm}[0pt]{$W$} &
				\raisebox{-.5mm}[0pt]{0.0555\;}
				& \raisebox{-.5mm}[0pt]{0.0548\;}
				& \raisebox{-.5mm}[0pt]{0.0934\;} &  \raisebox{-.5mm}[0pt]{0.0410\;}				
				\tabularnewline
				\raisebox{1.5mm}[0pt]{$\sigma^{2}$}
				&\raisebox{0mm}[0pt]{($p$--value)} &\raisebox{0mm}[0pt]{(0.4332)}
				&\raisebox{0mm}[0pt]{(0.4430)}
				&\raisebox{0mm}[0pt]{(0.1384)}
				&\raisebox{0mm}[0pt]{(0.6630)}
				\vspace{.2mm}\tabularnewline
				
				&\raisebox{-.75mm}[0pt]{$JB$} &
				\raisebox{-.75mm}[0pt]{2.1136\;}
				& \raisebox{-.75mm}[0pt]{1.8737\;}
				& \raisebox{-.75mm}[0pt]{3.1342\;} &  \raisebox{-.75mm}[0pt]{1.3905\; }					
				\tabularnewline
				&\raisebox{-.25mm}[0pt]{($p$--value)}
				&\raisebox{-.25mm}[0pt]{(0.337)}
				&\raisebox{-.25mm}[0pt]{(0.377)}
				&\raisebox{-.25mm}[0pt]{(0.194)}
				&\raisebox{-.25mm}[0pt]{(0.488)}
				\vspace{1.2mm}\tabularnewline
				\hline
				
				\raisebox{-.5mm}[0pt]{\small $^{*}p<0.05$}  & \raisebox{-.5mm}[0pt]{\small$^{**}p<0.01$\qquad}
			\end{tabular}

			}
		\end{table}

		\section{Application of the model}
		\label{sec:6}
		
		In this section, we describe a practical application of the GSB process of order $p=1$ in modelling dynamics of some financial series. For this purpose, we first observed the dynamics of the total values of trading 15 Serbian shares with the highest liquidity, integrated within the so--called BELEX15 financial index (Series A).
		This index was defined and methodologically processed at the end of September 2005. All its changes until the end of 2014 have been observed here,
		as the ``large" time series with a total of $T = 2\,330$ data. On the other hand, the total values of trading by the Minimum Price (MP) method on the Belgrade Stock Exchange (Series B) were also observed. It was chosen as an example of short time series with a total of (only) $ T = 127$ data.
		
		For both empirical data series (A and B), as a basic  financial series we observed the realization of  the log--volumes
		\[
		Y_t=\log\left(\sum_{j=1}^N S_t^{(j)}\cdot H_t^{(j)}\right),
		\]
		where $S_t^{(1)},\dots,S_t^{(N)}$ are the share prices (in Serbian dinars) and $H_t^{(1)},\dots,H_t^{(N)}$ are the volumes of trading, i.e., the number of shares which were traded on a certain day. In this way, the days of trading are used as successive data set $t=1,\dots, T$, and the estimates of the parameters were obtained from a realization of Split--MA(1) process $X_t=Y_t-Y_{t-1}$, where $X_1\stackrel{as}=0$.
		Estimated parameters' values from both series, along with the appropriate values of objective function $S_T^{(2)}$, are shown in the Tables \ref{Tab:3} and \ref{Tab:4}.\
		
	\begin{table}[htbp]
			\caption{Estimated values of the GSB parameters and error statistics of BELEX15 log--volumes (Series A).}\smallskip
			\label{Tab:3}{\small
				\begin{tabular}{lcccc}\hline\noalign{\smallskip}
				\raisebox{-4.5mm}[0pt]{Parameters}	& \raisebox{-3.5mm}[0pt]{Initial} & \multicolumn{3}{c}{\raisebox{-3.5mm}[0pt]{ECF estimates/weights}}
					\vspace{1.5mm}\tabularnewline
					\cline{3-5}
					&\raisebox{1.5mm}[0pt]{estimates} & \raisebox{-2mm}[0pt]{$g_1(u_1,u_2)$} & \raisebox{-2mm}[0pt]{$g_2(u_1,u_2)$} & \raisebox{-2mm}[0pt]{$g_3(u_1,u_2)$}
					\vspace{1.2mm}\tabularnewline
					\noalign{\smallskip}\hline\noalign{\smallskip}
					{$\hat{\rho}_T(1)$} & {-0.3973} &\, -- &\, -- &\, --
					\smallskip\\
					{${b}_c$} & {\; 0.6591} & {\; 0.5912} & {\; 0.4632} & {\; 0.5366}
					\smallskip\\
					{$c$} & {\;  0.5866} & {\; 0.6842} & {\; 0.3815} & {\; 0.5376}
					\smallskip\\
					{$\sigma^2$} & {\; 0.6467} & {\; 0.3752} & {\; 0.4381} & {\; 0.4645}
					\smallskip\\
				{${S}_T^{(2)}$} & {1.16E-5} & {1.62E-6} & {1.58E-6} & {8.68E-7}
					\\
					\noalign{\smallskip}\hline
				\end{tabular}}
			\end{table}\
			
For the initial values of parameters we took the estimates obtained by the method of moments, as in the previous simulations. After that, we obtained the ECF estimates of parameters for the previously described weights $g_k(u_1,u_2)$, $k=1,2,3$. Note that estimates of the first correlation of both series satisfy the inequality $-0,5< \hat{\rho}_{_T}(1) < 0$. Thus, according to (\ref{Eq.5.2}), the estimates of the parameter ${b}_c$, shown in the next  row of both tables, meet (the non-triviality) condition $0<b_c<1$. Also, the estimated values of parameters in the case of  ``small" Series B are more stable, but they have slightly higher values of estimated errors.\

	\begin{table}[htbp]
		\caption{Estimated values of the GSB parameters and error statistics of the MP method trading log--volumes (Series B).}\smallskip
		\label{Tab:4}{\small
			\begin{tabular}{lccccc}\hline\noalign{\smallskip}
		\raisebox{-4.5mm}[0pt]{Parameters}	& \raisebox{-3.5mm}[0pt]{Initial} & \multicolumn{3}{c}{\raisebox{-3.5mm}[0pt]{ECF estimates/weights}}
				\vspace{1.5mm}\tabularnewline
				\cline{3-5} 				&\raisebox{1.5mm}[0pt]{estimates} & \raisebox{-2mm}[0pt]{$g_1(u_1,u_2)$} & \raisebox{-2mm}[0pt]{$g_2(u_1,u_2)$} & \raisebox{-2mm}[0pt]{$g_3(u_1,u_2)$}
				\vspace{1.2mm}\tabularnewline
				\noalign{\smallskip}\hline\noalign{\smallskip}
				{$\hat{\rho}_T(1)$} & {-0.4771} & -- & -- & --
				\smallskip\\
				{${b}_c$} & {\; 0.9123} & {\; 0.8984} & {\; 0.9071} & {\; 0.9258}
				\smallskip\\
				{$c$} & {\; 6.7633} & {\; 6.0235} & {\; 6.4595} & {\; 7.5623}
				\smallskip\\
				{$\sigma^2$} & {\; 2.3195} & {\; 2.2468} & {\; 2.2876} & {\; 2.3716}
				\smallskip\\
				{${S}_T^{(2)}$} & {2.65E-5} & {4.21E-6} & {4.21E-6} & {2.35E-6}
				\\
				\noalign{\smallskip}\hline
			\end{tabular}}
		\end{table}
		
In Figure \ref{Fig.7} the empirical PDFs of both data series were compared with the PDFs obtained by fitting with the initial estimates, as well as with the ECF estimates of the Split--MA(1) process. As it can be easily seen, in both cases the ECF estimates provide better match to the empirical PDF. Note that the fitted PDF of the Series A (graph on the left) was estimated using the ECF procedure with a usual Gauss-Hermitian cubature, i.e., with the weight function $g_2(u_1,u_2)$ when $\gamma=2$. On the other hand, in the case of the Series B (graph on the right) cubature with the weight function $g_1(u_1,u_2)$ when $\gamma=1$ was used.

		\begin{figure}[hbtp]
			\begin{center}
				\includegraphics[width=1\textwidth]{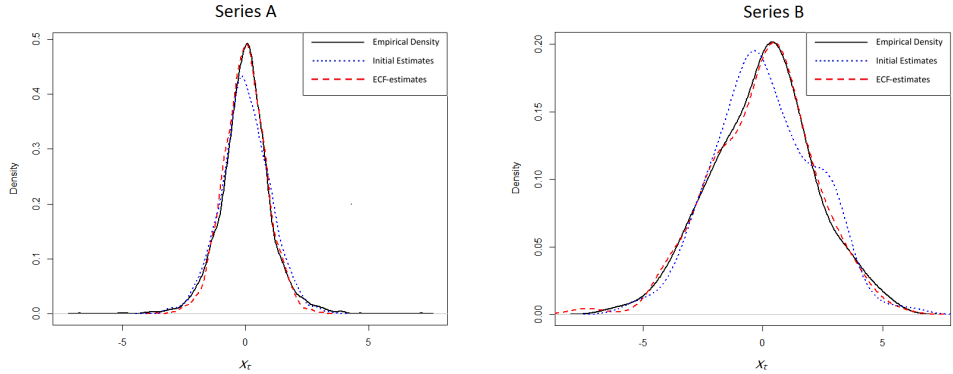}
				\caption{Empirical and fitted PDFs of the Split--MA(1) process.}\label{Fig.7}
			\end{center}
		\end{figure}

Finally, we explore the possibility of fitting the realizations of series  $(\eps_t)$ and $(m_t)$, as well as to compare them with the empirical series $(Y_t)$ and $(X_t)$. For this purpose, we apply the following iterative equations
\begin{equation}\label{PSB9}
\left\{%
\begin{array}{ll}
\eps_t=y_t-m_t \\
m_{t}=m_{t-1}+\eps_{t-1}I(\eps_{t-2}^2>\hat{c} )
\end{array}%
\right., \qquad t=2,\dots,T
\end{equation}
where $\hat{c}$ is the estimate of the critical value ${c}$,  obtained by using the previous ECF procedure, with cubature weight parameter  $\gamma=2$ (Series A), and $\gamma=1$ (Series B). On the other hand, we use
$\eps_1=\eps_0\stackrel{as}= 0$ and $m_1=\overline{y}_T$ as initial values of the iterative procedure (\ref{PSB9}),
where $\overline{y}_T$ is the empirical mean
of the series $(Y_t)$.	

A correlation between the log--volumes and the martingale means can also be seen in Figure \ref{Fig.8} (graph on the left). Obviously, the high correlation of the Series A implies greater fluctuations in the dynamics of martingale means series. In that way, it points to the empathic presence of the ``large shocks" in dynamics of the log--volumes. 	
On the other hand, it is obvious that Series B has a slightly more prominent stability and a smaller presence of the emphatic fluctuations in its dynamics. In the same Figure (graphs on the right) graphs of the realizations of the Split--MA(1) process $(X_t)$ and the innovations $(\eps_t)$ are shown.
					
\begin{figure}[htbp]
	\begin{center}
		\includegraphics[width=.9\textwidth]{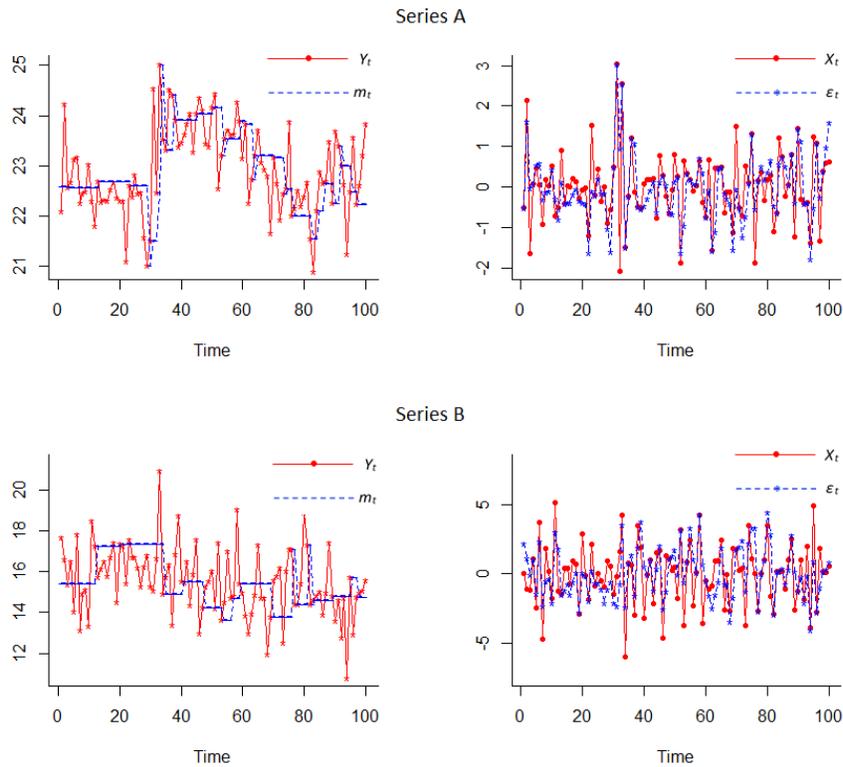}
		\caption{Comparative graphs of the actual and modeled data.}\label{Fig.8}
	\end{center}
\end{figure}

\section{Conclusion}
\label{sec:7}

In recent years, various modifications of the STOPBREAK process have been successfully applied to describe dynamics of time series with emphatic and permanent fluctuations. The previous statistical analysis confirms and justifies such a possibility in the case of the GSB process, where the Gaussian distribution for the innovations is assumed. In this paper, the GSB process has been used as a stochastic model in order to describe the behavior of the market capitalizations on the Serbian stock market. However, with certain modifications, this process can also be applied for estimation of similar time series (financial or any other).
	
\section*{Acknowledgements}
The authors would like to thank Belgrade Stock Exchange for the consignment of datasets and the constructive information regarding them. Authors are also very grateful to an anonymous referee for pointing out several mistakes in a preliminary version, together with valuable suggestions and comments that have allowed us to  improve the paper.

\end{document}